\documentclass[a4paper,10pt,draft]{amsart}

\usepackage{amsmath,amssymb,amsthm}
\usepackage[alphabetic, abbrev]{amsrefs}
\usepackage{enumerate}
\usepackage[all]{xy}
\newdir{ >}{{}*!/-5pt/@{>}} 
\SelectTips{cm}{} 
\usepackage{hyperref}

\numberwithin{equation}{section}

\newtheorem{theorem}[subsection]{Theorem}
\newtheorem{corollary}[subsection]{Corollary}
\newtheorem{lemma}[subsection]{Lemma}
\newtheorem{proposition}[subsection]{Proposition}
\theoremstyle{definition}
\newtheorem{definition}[subsection]{Definition}

\newtheorem{remark}[subsection]{Remark}

\newtheorem{construction}[subsection]{Construction}

\newcommand{\bC}{\mathbb{C}}

\newcommand{\bZ}{\mathbb{Z}}
\newcommand{\bF}{\mathbb{F}}

\newcommand{\cC}{\mathcal{C}}
\newcommand{\cD}{\mathcal{D}}
\newcommand{\cE}{\mathcal{E}}

\newcommand{\cI}{\mathcal{I}}

\newcommand{\cN}{\mathcal{N}}

\DeclareMathOperator{\hocolim}{hocolim}

\DeclareMathOperator{\colim}{colim}
\DeclareMathOperator{\const}{const}

\DeclareMathOperator{\Tot}{Tot}
\DeclareMathOperator{\sh}{sh}
\DeclareMathOperator{\srep}{srep}

\newcommand{\op}{{\mathrm{op}}}
\newcommand{\id}{{\mathrm{id}}}
\newcommand{\tensor}{\otimes}

\newcommand{\iso}{\cong}
\newcommand{\concat}{\sqcup}
\newcommand{\bld}[1]{{\mathbf{#1}}}

\newcommand{\chk}{\mathrm{Ch}_k}
\newcommand{\schk}{\mathrm{sCh}_k}
\newcommand{\chQ}{\mathrm{Ch}_{\mathbb Q}}
\newcommand{\chIk}{\mathrm{Ch}^{\mathcal I}_k}
\newcommand{\cdga}{\mathrm{cdga}}
\newcommand{\cdgaQ}{\cdga_{\mathbb Q}}
\newcommand{\CchIk}{\mathrm{Ch}^{\mathcal I}_k[\cC]}
\newcommand{\CchIZ}{\mathrm{Ch}^{\mathcal I}_{\mathbb Z}[\cC]}
\newcommand{\Mod}{\mathrm{Mod}}
\newcommand{\sMod}{\mathrm{sMod}}
\newcommand{\Echk}{\mathrm{Ch}_k[\cE]}
\newcommand{\EchIk}{\mathrm{Ch}_k^{\mathcal I}[\cE]}
\newcommand{\EchZ}{\mathrm{Ch}_{\mathbb Z}[\cE]}
\newcommand{\EcchZ}{\mathrm{Ch}_{\mathbb Z}[\cE^{\mathrm{cof}}]}
\newcommand{\Ecchk}{\mathrm{Ch}_k[\cE^{\mathrm{cof}}]}

\newcommand{\sset}{\mathrm{sSet}}
\newcommand{\cat}{\mathrm{cat}}

\newcommand{\ie}{\emph{i.e.}}
\newcommand{\hot}{\hat{\otimes}}
\newcommand{\ra}{\rightarrow}
\newcommand{\arxivlink}[1]{\href{http://arxiv.org/abs/#1}{\texttt{arXiv:#1}}}

\usepackage{ifdraft}

\ifdraft{%
  \newcommand{\Bemerkung}[1]{{\marginpar{\hspace{0.2\marginparwidth}\rule{0.6\marginparwidth}{0.75mm}\hspace{0.2\marginparwidth}}\noindent\bfseries[#1]}}
}{
  \newcommand{\Bemerkung}[1]{}
}

\title{A strictly commutative model for the cochain algebra of a space}

\author{Birgit Richter} 
\address{Universit\"at Hamburg, Germany}
\email{birgit.richter@uni-hamburg.de}
\author{Steffen Sagave} \address{Radboud Universiteit Nijmegen, The
  Netherlands}  
\email{s.sagave@math.ru.nl}

\date{\today}

\begin{document}
\begin{abstract}
  The commutative differential graded algebra $A_{\mathrm{PL}}(X)$ of
  polynomial forms on a simplicial set $X$ is a crucial tool in
  rational homotopy theory. In this note, we construct an integral
  version $A^{\cI}(X)$ of $A_{\mathrm{PL}}(X)$. Our approach uses
  diagrams of chain complexes indexed by the category of finite sets
  and injections $\cI$ to model $E_{\infty}$ differential graded
  algebras by strictly commutative objects, called commutative
  $\cI$-dgas. We define a functor $A^{\cI}$ from simplicial sets to
  commutative $\cI$-dgas and show that it is a commutative lift of the
  usual cochain algebra functor. In particular, it gives rise to a new
  construction of the $E_{\infty}$ dga of cochains.

  The functor $A^{\cI}$ shares many properties of $A_{\mathrm{PL}}$,
  and can be viewed as a generalization of $A_{\mathrm{PL}}$ that
  works over arbitrary commutative ground rings. Working over the
  integers, a theorem by Mandell implies that $A^{\cI}(X)$ determines
  the homotopy type of $X$ when $X$ is a nilpotent space of finite
  type.
\end{abstract}

\maketitle

\section{Introduction}
Determining the homotopy type of a topological space is a difficult task in
general. One possibility of simplifying the problem is to aim for algebraic
models of spaces, so that the study of homotopy types reduces to an algebraic
question. If one is interested in the homotopy type of a rational nilpotent
space of finite type, then this is possible, and the Sullivan cochain algebra
is such an algebraic model: the algebra of rational singular cochains of a
space, $C(X;\mathbb Q)$, is quasi-isomorphic to the commutative differential
graded algebra (cdga) $A_{\mathrm{PL}}(X)$ of polynomial forms on $X$, which
is a very powerful tool in rational homotopy
theory~\cite{sullivan,Bousfield-G_rational}. The functor
$A_{\mathrm{PL}}$ has a contravariant adjoint, called the
Sullivan realization in \cite[\S 17]{FHT}. With the help of this adjoint pair of
functors one can determine the homotopy type of rational nilpotent spaces of
finite type (see \cite[Chapter~9]{Bousfield-G_rational}, or
\cite[Theorem 1.25]{hess} for the simply connected case). 

For a general commutative ring $k$ the cochains $C(X;k)$ on a space $X$ with
values in 
$k$ form a differential graded algebra whose cohomology is the
singular cohomology $H^*(X;k)$ of~$X$. The multiplication of $C(X;k)$
induces the cup product on $H^*(X;k)$. However, for general  $k$, there is
no cdga which is quasi-isomorphic to $C(X;k)$, for example because the
Steenrod operations witness the non-commuta\-ti\-vi\-ty of
$C(X;\bF_p)$. So it seems that we cannot hope for a strictly commutative
model for the cochains of a space that determines the homotopy type.
However, $C(X;k)$ is always commutative up to coherent
homotopy. This can be encoded using the language of operads
\cite{may-gils}: the multiplication of $C(X;k)$ extends
to the action of an $E_{\infty}$ operad in chain complexes turning $C(X;k)$
into an $E_{\infty}$ dga.

This gives rise to an algebraic model for the homotopy type of a space by a
result of Mandell. He shows that the cochain functor $C(-;\mathbb Z)$ to
$E_{\infty}$ dgas classifies nilpotent spaces of finite type up
to weak equivalence~\cite[Main Theorem]{Mandell_cochains-homotopy-type}. But
here, the algebraic model consists of the cochain algebra together with its
$E_\infty$-algebra structure, so this algebraic model is rather involved. 

One can describe homotopy coherent commutative multiplications on
chain complexes using diagram categories instead of operads. Let $\cI$
be the category with objects the finite sets
$\bld{m}= \{1\,\dots, m\}, m\geq 0$, with the convention that
$\bld{0}$ is the empty set. Morphisms in $\cI$ are the
injections. Concatenation in $\cI$ and the tensor product of chain
complexes of $k$-modules give rise to a symmetric monoidal
product~$\boxtimes$ on the category $\chIk$ of $\cI$-diagrams in
$\chk$. A \emph{commutative $\cI$-dga} is a commutative monoid in
$(\chIk,\boxtimes)$ or, equivalently, a lax symmetric monoidal functor
$\cI \to \chk$. Equipped with suitable model structures, the category
of commutative $\cI$-dgas, $\CchIk$, is Quillen equivalent to the
category of $E_{\infty}$ dgas~\cite[\S 9]{Richter-S_algebraic}. This
is analogous to the situation in spaces, where commutative monoids in
$\cI$-diagrams of spaces are equivalent to $E_{\infty}$
spaces~\cite[\S 3]{Sagave-S_diagram}.

Chasing the $E_{\infty}$~dga of cochains $C(X;k)$ on a space $X$
through the chain of Quillen equivalences relating $E_{\infty}$~dgas
and commutative $\cI$-dgas shows that $C(X;k)$ can be represented by a
commutative $\cI$-dga. The purpose of this paper is to construct a
direct point set level model $A^{\cI}(X)$ for the quasi-isomorphism
type of commutative $\cI$-dgas determined by $C(X;k)$ that should be
viewed as an integral generalization of $A_{\mathrm{PL}}(X)$. Despite
the fact that $A_{\mathrm{PL}}(X)$ was introduced more than 40 years
ago and has been widely studied, it appears that a direct integral
counterpart was neither known nor expected to exist.

If $E$ is a commutative $\cI$-dga, then its Bousfield--Kan homotopy
colimit $E_{h\cI}$ has a canonical action of the Barratt--Eccles
operad, which is an $E_{\infty}$ operad built from the symmetric
groups. The commutative $\cI$-dga $A^{\cI}(X)$ thus gives rise to an
$E_{\infty}$ dga $A^{\cI}(X)_{h\cI}$ which can be compared to the
usual cochains without referring to model structures.

\begin{theorem}\label{thm:E-infty-comparison-intro}
  The contravariant functors $X \mapsto A^{\cI}(X)_{h\cI}$ and
  $X \mapsto C(X;k)$ from simplicial sets to $E_{\infty}$ dgas are
  naturally quasi-isomorphic.
\end{theorem}

We prove the theorem using Mandell's uniqueness result for cochain
theories \cite[Main Theorem]{Mandell_cochain-multiplications}.

Since our definition of $A^{\cI}$ does not rely on the existing
constructions of $E_{\infty}$ structures on cochains, the theorem
implies that our approach provides an alternative model of the
$E_{\infty}$ dga $C(X;k)$, namely $A^{\cI}(X)_{h\cI}$ with its
canonical action of the Barratt-Eccles operad. If $k$ is a field of
characteristic $0$, 
then there is a natural quasi-isomorphism
$A^{\cI}(X)_{h\cI} \to A_{\mathrm{PL}}(X)$ relating our approach to
the classical polynomial forms (see Theorem \ref{thm:char0}).

The passage through commutative $\cI$-dgas has the advantage that it
provides a rather simple $E_{\infty}$ model $A^{\cI}(X)_{h\cI}$ for
the cochain algebra of a space $X$. In contrast, the existing
constructions of $E_{\infty}$ structures on the standard model for the
cochain algebra are involved: based on work of
Hinich--Schechtman~\cite{Hinich-S_homotopy-limit}, Mandell~\cite[\S
5]{Mandell_cochain-multiplications} lifts the action of the acyclic
Eilenberg--Zilber operad to the action of an actual $E_{\infty}$
operad. McClure and Smith generalize Steenrod’s cup-$i$-products to
multivariable operations that give the cochains of a space the
structure of an $E_\infty$-algebra via the action of the surjection
operad \cite{McClure-Smith_multivariable}. 
Berger and Fresse~\cite{Berger-F_combinatorial} use elaborate
combinatorial arguments to define an action of the Barratt--Eccles
operad that extends the action of the surjection operad. 
Another approach 
to capture the commutativity of $C(X;k)$ has been pursued by
Karoubi~\cite{Karoubi} who introduces a notion of
\emph{quasi-commutative} dgas that is based on a certain reduced
tensor product, constructs a quasi-commutative model for the cochains,
and uses Mandell's results to relate it to ordinary cochains.

Since it is often easier to work with strictly commutative objects
rather than $E_{\infty}$ objects, we also expect that the commutative
$\cI$-dga $A^{\cI}(X)$ will be a useful replacement of the
$E_{\infty}$ dga $C(X;k)$ in applications. For instance, iterated bar
constructions for $E_\infty$ algebras as developed in
\cite{fresse-bar} are rather involved whereas iterated bar
construction for commutative monoids are straightforward. Commutative
$\cI$-dgas are tensored over simplicial sets whereas enrichments for  
$E_\infty$ monoids are more complicated because the coproduct is not
just the underlying monoidal product. This allows for constructions
such as higher order Hochschild homology \cite{pira-hodge} for
commutative $\cI$-dgas.  

Writing $A^{\cI}(X;\mathbb Z)$ for $A^{\cI}(X)$ when working over
$k = \mathbb Z$, Theorem~\ref{thm:E-infty-comparison-intro} leads to the following reformulation of the main theorem of
Mandell~\cite{Mandell_cochains-homotopy-type} that highlights the
usefulness of $A^{\cI}$:
\begin{theorem}\label{thm:mandell-intro}
Two finite type nilpotent spaces $X$ and $Y$ are weakly equivalent if and 
only if $A^{\cI}(X;\mathbb Z)$ and $A^{\cI}(Y;\mathbb Z)$ are weakly equivalent in $\CchIZ$. 
\end{theorem} 
\subsection{Outline of the construction}\label{subsec:outline_constr}
Our  chain complexes are homologically graded so that
cochains are concentrated in non-positive degrees. We model spaces
by simplicial sets and consider the singular complex of a topological space
if necessary.

The functor $A_{\mathrm{PL}}\colon \sset^{\op} \to \cdgaQ$ of
polynomial forms used in rational homotopy theory (see e.g.~\cite[\S
1]{Bousfield-G_rational}) motivates our definition of $A^{\cI}$. We
recall that $A_{\mathrm{PL}}$ arises by Kan extending the functor
$A_{\mathrm{PL},\bullet} \colon \Delta^{\op} \to \cdgaQ$ sending $[p]$
in $\Delta$ to the algebra of polynomial differential forms 
\begin{equation} \label{eq:sullivanapl}
  A_{\mathrm{PL},p} =
\Lambda(t_0,\dots,t_p;dt_0,\dots,dt_p)/(t_0+\dots+t_p=1,
dt_0+\dots+dt_p = 0).\end{equation}
Here $\Lambda$ is the free graded commutative algebra over
$\mathbb Q$, the generators $t_i$ have degree $0$, and the $dt_i$ have
degree $-1$ (in our homological grading). Setting $d(t_i) = dt_i$
extends to a differential that turns $A_{\mathrm{PL},q}$ into a
commutative dga, and addition of the $t_i$ and insertion of $0$ define
the simplicial structure of $A_{\mathrm{PL},\bullet}$.

The topological standard $p$-simplex can be written as
\[ \Delta^p = \{(t_0,\ldots, t_p), t_0+\dots+t_p=1, t_i \geq 0\} \]
and as
\[ \Delta^p = \{(x_0, x_1, \ldots, x_p, x_{p+1}), 0=x_0 \leq x_1 \leq \ldots
  \leq x_p \leq x_{p+1}=1\}. \] 
Setting $x_i =
t_0 + \ldots + t_{i-1}$ for $1 \leq i \leq p$ yields an isomorphism
\begin{equation} \label{eq:sullivanaplbar}
  A_{\mathrm{PL},p} \cong \Lambda(x_1,\ldots, x_p, dx_1,\ldots,
  dx_p) \cong \Lambda(x_1, dx_1) \otimes \ldots \otimes \Lambda(x_p, dx_p)
\end{equation}
and this simple but crucial trick gives rise to the following
reformulation:  let $\bC D^0$ be the free commutative $\mathbb Q$-dga on the chain
complex $D^0$ with $(D^0)_i=0$ if $i\neq 0,-1$ and
$d_0\colon (D^0)_0 \to (D^0)_{-1}$ being $\mathrm{id}_{\mathbb
  Q}$. Moreover, let $S^0$ in $\chQ$ be the monoidal unit, \ie, the
chain complex with a copy of $\mathbb Q$ concentrated in degree
$0$. Sending $1 \in (\bC D^0)_0$ to either $0$ or $1$ in $\mathbb Q$
defines two commutative $\bC D^0$-algebra structures on $S^0$ that we
denote by $S^0_0$ and $S^0_1$. We argue in Section~\ref{subsec:A_PL}
that the simplicial $\mathbb Q$-cdga $A_{\mathrm{PL},\bullet}$ is
isomorphic to the two sided bar construction
\begin{equation}\label{eq:A_PL}  
B_{\bullet}(S^0_0,\bC D^0, S^0_1) = \left([p]
  \mapsto S^0_0 \tensor (\bC D^0)^{\tensor p} \tensor
  S^0_1\right)  
\end{equation}
whose face maps are provided by the algebra structures on $S^0_1$ and
$S^0_0$ and the multiplication of $\bC D^0$, and whose degeneracy
maps are induced by the unit of $\bC D^0$.

While polynomial differential forms appear to have no obvious
counterpart in commutative $\cI$-dgas, their description in terms of
a two sided bar construction easily generalizes to commutative $\cI$-dgas
over an arbitrary commutative ground
ring~$k$. For this we consider the left adjoint
\[ \bC F_{\bld{1}}^{\cI} \colon \chk \to \CchIk, \qquad A\mapsto
  \left(\bld{m}\mapsto \textstyle \bigoplus_{s \geq 0} \left(\left(
        \textstyle \bigoplus_{\cI(\bld{1}^{\concat s},\bld{m})}
        A^{\tensor s}\right)/\, \Sigma_s\right)\right)\] to the
evaluation of a commutative $\cI$-dga at the object $\bld{1}$ in $\cI$
and recall that the unit $U^{\cI}$ in $\chIk$ is the constant
$\cI$-diagram on the unit $S^0$ in $\chk$. As above, we form $\bC
F_{\bld{1}}^{\cI} D^0$, observe that $U^{\cI}$ gives rise to two
commutative $\bC F_{\bld{1}}^{\cI} D^0$ algebras $U^{\cI}_0$ and
$U^{\cI}_1$, and define $A^{\cI}_{\bullet} \colon \Delta^{\op} \to
\CchIk$ to be the two sided bar construction
\[ B_{\bullet}(U^{\cI}_0,\bC F^{\cI}_{\bld{1}}(D^0) ,U^{\cI}_1) =
\left( [p] \mapsto U^{\cI}_0 \boxtimes (\bC
  F^{\cI}_{\bld{1}}(D^0))^{\boxtimes p} \boxtimes U^{\cI}_1 \right).\]
At this point it is central to work with strictly commutative
objects since the multiplication map of an $E_{\infty}$ object is
typically not an $E_{\infty}$ map. It is also important to use
$\bld{1}$ rather than $\bld{0}$ in the above left adjoint since this
ensures that $A^{\cI}_p(\bld{m})$ is contractible. This is related to
J. Smith's insight that one has to use \emph{positive} model
structures for commutative symmetric ring spectra. Using that $\boxtimes$ is the coproduct in commutative $\cI$-dgas, we get an isomorphism 
\[ A^{\cI}_p \iso  \bC F^{\cI}_{\bld{1}}(D^0 \oplus \dots \oplus
  D^0) \]
identifying the simplicial degree $p$ part of $A^{\cI}_\bullet$ with a
free commutative $\cI$-dga on $p$ generators. We also describe the
simplicial structure maps of $A^{\cI}_\bullet$ in terms of these
generators (see Section~\ref{subsec:I-version-polynomial}). 

Via Kan extension and restriction along the canonical functor
$\Delta^{\op} \to \sset^{\op}$, this $A^{\cI}_{\bullet}$ gives
rise to functors $A^{\cI} \colon \sset^{\op} \to \CchIk$ and
$\langle-\rangle_\cI \colon \CchIk^{\op} \to \sset$ (see Definition
\ref{def:aiandki}).
More explicitly, the
evaluation of $A^{\cI}(X)$ at $\cI$-degree $\bld{m}$ and chain complex
level $q$ is the $k$-module of simplicial set morphisms
$\sset(X,A^{\cI}_{\bullet}(\bld{m})_q)$. For every
$E$ in $\CchIk$, we set $\langle E\rangle_\cI = \CchIk(E, A^\cI_\bullet)$.
The functors $A^{\cI}$ and
$\langle -\rangle_\cI$ are contravariant right adjoint in the sense that there
are natural isomorphisms $\CchIk(E, A^{\cI}(X))\iso \sset(X,
\langle E \rangle_\cI)$. They are integral analogues of the functor of
polynomial forms and of the Sullivan realization functor. 

\subsection{Homotopical analysis of \texorpdfstring{$A^{\cI}$}{A-I}}
We equip simplicial sets with the standard model structure and the
category of commutative $\cI$-dgas $\CchIk$ with the \emph{descending
$\cI$-model structure} making it Quillen equivalent to $E_{\infty}$
dgas (see Section~\ref{sec:model_str} for details). 
\begin{theorem}\label{thm:AI-KI-adjunctions-intro}
  Both $A^{\cI}$ and $\langle -\rangle_\cI$ send cofibrations to fibrations and
  acyclic cofibrations to acyclic fibrations.  They induce functors 
  on the corresponding homotopy categories
  $\mathbb R\langle - \rangle_\cI \colon \mathrm{Ho}(\CchIk)^{\op}\to
   \mathrm{Ho}(\sset)$
   and
   $\mathbb RA^{\cI}\colon \mathrm{Ho}(\sset)^{\op}\to
   \mathrm{Ho}(\CchIk)$ that are related by a natural isomorphism
   \[\mathrm{Ho}(\CchIk)(E,\mathbb RA^{\cI}(X)) \iso
     \mathrm{Ho}(\sset)(X,\mathbb R\langle E\rangle_\cI).\]  
 \end{theorem}
Here, $\mathbb R(-)$ indicates that we right-derive the functors on $(\CchIk)^{\op}$ and $\sset^{\op}$. So no fibrant replacement is necessary before applying $A^{\cI}$ since all simplicial sets are cofibrant, while a cofibrant replacement in $\CchIk$ is necessary to derive $\langle - \rangle_{\cI}$.

A similar result for $A_{\mathrm{PL}}\colon \sset^{\op} \to \cdgaQ$
has been established by Bousfield--Gugen\-heim~\cite[\S
8]{Bousfield-G_rational}.  Mandell~\cite[\S
4]{Mandell_cochain-multiplications} constructed an analogous adjunction
between simplicial sets and $E_{\infty}$ dgas using the $E_{\infty}$ structure
on cochains as input. The functor of homotopy categories $\mathbb
RA^{\cI}$ fails to be full for the same reason as its counterpart
studied by Mandell (see the discussion after~\cite[Theorem
0.2]{Mandell_cochains-homotopy-type}). One of the referees of this
paper raised the interesting question whether there exists a
modification of our diagrammatic approach that remedies this
shortcoming.

Since all simplicial sets are cofibrant, the statement of
Theorem~\ref{thm:AI-KI-adjunctions-intro} implies that each
$A^{\cI}(X)$ is descending $\cI$-fibrant. Writing $\cI_{+}$ for the
full subcategory of $\cI$ on objects $\bld{m}$ with $|\bld{m}|\geq 1$,
this means that for each morphism $\bld{m} \to \bld{n}$ in $\cI_+$ and
each $q \geq -|\bld{m}|$, the induced map
$H_q(A^{\cI}(X)(\bld{m})) \to H_q(A^{\cI}(X)(\bld{n}))$ is an
isomorphism. Hence each chain complex $A^{\cI}(X)(\bld{m})$ with
$\bld{m}$ in $\cI_+$ captures the cohomology groups of $X$ up to
degree $|\bld{m}|$. This is the maximal information to be expected
from $A^{\cI}(X)(\bld{m})$ as it is a chain complex concentrated
in degrees between $0$ and $-|\bld{m}|$. Since the descending $\cI$-model
structure is the left Bousfield localization of a descending level
model structure, it also follows that weak homotopy equivalences
$X \to Y$ induce isomorphisms
$H_q(A^{\cI}(Y)(\bld{m})) \to H_q(A^{\cI}(X)(\bld{m}))$ if $\bld{m}$
is in~$\cI_+$ and $q \geq -|\bld{m}|$.

Analogous to the corresponding statement about $A_{\mathrm{PL}}$, the
proof of the theorem is based on the observation that the simplicial
sets $A^{\cI}_{\bullet}(\bld{m})_q$ are contractible for a fixed $\bld{m}$ in
$\cI_+$ and a fixed chain level $q$ with $0 \geq q \geq -|\bld{m}|$.

\begin{remark}
  After a first version of the present manuscript was made available,
  the authors learned from Dan Petersen that he recently found another
  construction of a commutative $\cI$-dga that models the cochain
  algebra of a space~\cite{Petersen_configuration}. His approach
  applies to locally contractible topological spaces, uses sheaf
  cohomology, and has applications in the study of configuration
  spaces.
\end{remark}

\subsection{Notations and conventions} Throughout the paper, $k$
denotes a commutative ring with unit, and $\chk$ denotes the category
of unbounded homologically graded chain complexes of $k$-modules. For
$q\in \mathbb Z$, we as usual write $S^q$ for the chain complex with
$k$ concentrated in degree $q$, and $D^q$ for the chain complex with
$(D^q)_i=k$ if $i\in\{q,q-1\}$, with $(D^q)_i=0$ for all other $i$,
and with $d_q = \mathrm{id}_k$.

\subsection{Organization} In Section~2 we study homotopy colimits of
commutative $\cI$-dgas. Section~3 provides the construction of the
functor $A^{\cI}$. We review model structures on $\cI$-chain complexes
and commutative $\cI$-dgas in Section~4. In Section~5 we establish the
homotopical properties of $A^{\cI}$, prove a comparison to the usual
cochains disregarding multiplicative structures, and prove
Theorem~\ref{thm:AI-KI-adjunctions-intro}. In the final Section~6, we
prove the $E_{\infty}$ comparison from
Theorem~\ref{thm:E-infty-comparison-intro} as
Theorem~\ref{thm:E-infty-comparison} and explain how to derive
Theorem~\ref{thm:mandell-intro}.

\subsection{Acknowledgments} The authors thank the referees for useful
comments on an earlier version of this paper and Josefien Kuijper for
bringing an error in a previous version of
Lemma~\ref{lem:AI-contractible} to our attention.
\section{Homotopy colimits of \texorpdfstring{$\cI$}{I}-chain complexes}
Let $\cI$ be the category with objects the finite sets
$\bld{m}=\{1,\dots,m\}$ for $m\geq 0$ and with morphisms the injective
maps. In this section we study multiplicative properties of the
homotopy colimit functor for $\cI$-diagrams of chain complexes.

\begin{definition} An \emph{$\cI$-chain complex} is a functor $\cI \to
  \chk$, and $\chIk$ denotes the resulting functor category.
\end{definition}

For each $\bld{m}$ in $\cI$ there is an adjunction
$F^{\cI}_{\bld{m}}\colon \chk \rightleftarrows \chIk  \colon
\mathrm{Ev}_{\bld{m}}$ with right adjoint the evaluation functor
$\mathrm{Ev}_{\bld{m}}(P) = P(\bld{m})$ and left adjoint  
\begin{equation}\label{eq:free-I-chain-complex}
  F^{\cI}_{\bld{m}}\colon \chk \to \chIk, \quad A \mapsto 
\left(\bld{n} \mapsto \textstyle\bigoplus_{\cI(\bld{m},\bld{n})} A
\right)\ .
\end{equation}
The functor $F^{\cI}_{\bld{0}}$ is isomorphic to the
constant functor since $\bld{0}$ is initial in $\cI$. 
\subsection{Homotopy colimits}
Our next aim is to define Bousfield--Kan style homotopy colimits for
$\cI$-diagrams of chain complexes. For the subsequent multiplicative
analysis, we fix notation and conventions about bicomplexes.
\begin{definition} Let $\chk(\chk)$ be the category of chain complexes
  in $\chk$. Its objects are $\bZ \times \bZ$-graded $k$-modules
  $(Y_{p,q})_{p,q \in \bZ}$ with $k$-linear \emph{horizontal 
    differentials}, $d_h \colon Y_{p,q} \ra Y_{p-1,q}$, and $k$-linear
  \emph{vertical differentials}, $d_v \colon Y_{p,q} \ra Y_{p,q-1}$,
  such that
  \[ d_h \circ d_h = 0 = d_v \circ d_v \text{ and } d_v \circ d_h =
    d_h \circ d_v.\]
  A morphism $g \colon Y \ra Z$ in $\chk(\chk)$ is a
  family $(g_{p,q} \colon Y_{p,q} \ra Z_{p,q})_{p,q \in \bZ \times
    \bZ}$ of $k$-linear maps that commute with the horizontal and
  vertical differentials, \ie,
  \[ d_h \circ g_{p,q} = g_{p-1,q} \circ d_h \text{ and } d_v \circ
  g_{p,q} = g_{p,q-1} \circ d_v\] for all $p,q \in \bZ$.
\end{definition}
Since we require horizontal and vertical differentials to commute, an
additional sign is needed to form the total complex:
\begin{definition}
  Let $Y$ be an object in $\chk(\chk)$. Its \emph{associated total
    complex} $\Tot(Y)$ is the chain complex with $\Tot(Y)_n =
  \bigoplus_{p+q=n} Y_{p,q}$ in chain degree $n \in \bZ$ and with
  differential $d_{\Tot}(y) = d_h(y) + (-1)^pd_v(y)$ for every
  homogeneous $y \in Y_{p,q}$.
\end{definition}

Let $\schk$ be the category of simplicial objects in $\chk$. 
\begin{definition} \label{def:moorechains}
For $A \in \schk$ we denote by $C_*(A)$ the chain complex in chain
complexes with $(C_*(A))_{p,q} = A_{p,q}$. We define the horizontal
differential on $C_*(A)$,  $d_h\colon A_{p,q} \ra A_{p-1,q}$, as 
\[ d_h = \sum_{i=0}^p (-1)^i d_i\] 
where the $d_i$ are the simplicial face maps of $A$. The vertical
differential on $C_*(A)$ is given by the differential $d^A$ on $A$. 

As the $d_i$'s commute with $d^A$, this gives indeed a chain
complex in chain complexes whose horizontal part is concentrated in
non-negative degrees. 
\end{definition}
\begin{construction}\label{constr:hocolim_I}
  Let $P\colon \cI \to \chk$ be an $\cI$-chain complex.  The
  \emph{simplicial replacement} of $P$ is the
  simplicial chain complex $\srep(P) \colon \Delta^{\op} \to
  \mathrm{Ch_k}$ given in simplicial degree $[p]$ by
  \[ \srep(P)[p] = \bigoplus_{(\bld{n_0} \xleftarrow{\alpha_1}
    \dots \xleftarrow{\alpha_p} \bld{n_p}) \in N(\cI)_p} P(\bld{n_p}).\] 
The last face map sends the copy of $P(\bld{n_p})$ indexed by
  $(\alpha_1,\dots, \alpha_p)$ via $P(\alpha_p)$ to the copy of
  $P(\bld{n_{p-1}})$ indexed by $(\alpha_1,\dots,\alpha_{p-1})$. The
  other face and degeneracy maps are induced by the identity on
  $P(\bld{n_p})$ and corresponding simplicial structure maps of the
  nerve $\cN(\cI)$ of $\cI$.  

  The homotopy colimit functor $(-)_{h\cI}\colon \chIk \to \chk$ is
  defined by
  \[ P_{h\cI} = \mathrm{Tot}\; C_*(\srep(P)).\]
\end{construction}

A bicomplex spectral sequence argument shows that
$P_{h\cI} \to Q_{h\cI}$ is a quasi-isomorphism if each
$P(\bld{m}) \to Q(\bld{m})$ is a quasi-isomorphism.  There is a
canonical map $P_{h\cI} \to \colim_{\cI}P$, and one can show by cell
induction that it is a quasi-isomorphism if $P$ is cofibrant in the
projective level model structure on $\chIk$. Together this shows that
$P_{h\cI}$ is a model for the homotopy colimit of $P$. A more elaborate
argument that shows that $P_{h\cI}$ is a corrected homotopy colimit
can be found in \cite{rg}. A version of the above homotopy colimit for
functors with values in modules can be found in \cite[Definition
3.13]{davislueck}.

\subsection{Commutative \texorpdfstring{$\cI$}{I}-dgas} 
The ordered concatenation of ordered sets $\bld{m}\concat\bld{n} =
\bld{m+n}$ equips $\cI$ with a symmetric strict monoidal
structure that has $\bld{0}$ as a strict unit and the block
permutations as symmetry isomorphisms. If $P,Q \colon \cI\to \chk$ are
$\cI$-chain complexes, then the left Kan extension of
\[ \cI \times \cI \xrightarrow{P \times Q} \chk \times \chk
\xrightarrow{\tensor} \chk \] 
along $\concat \colon \cI \times \cI \to
\cI$ provides an $\cI$-chain complex $P\boxtimes Q$. This
defines a symmetric monoidal product $\boxtimes$ on $\chIk$, the Day
convolution product,  with unit
the constant $\cI$-diagram $U^{\cI}= F_0^{\cI}(S^0)$. 
\begin{definition}
  A \emph{commutative $\cI$-dga} is a commutative monoid in
  $(\chIk, \boxtimes, U^{\cI})$, \ie, a lax symmetric monoidal functor
  $(\cI, \concat, \bld{0}) \to (\chk, \tensor, S^0)$. The
  resulting category of commutative $\cI$-dgas is denoted by $\CchIk$.
\end{definition}

We write $\bC\colon \chIk \rightleftarrows \CchIk \colon U$ for
the adjunction with right adjoint the forgetful functor and left adjoint
the free functor $\bC$ given by
\begin{equation}\label{eq:free-commutative} \bC(P) = \textstyle
  \bigoplus_{s \geq 0}P^{\boxtimes s}/ \Sigma_s. 
\end{equation}
The definition of $\boxtimes$ as a left Kan extension implies
the existence of a natural isomorphism $F_{\bld{n_1}}^{\cI}(A^1)
\boxtimes F_{\bld{n_2}}^{\cI}(A^2) \iso F_{\bld{n_1\concat
    n_2}}^{\cI}(A^1\tensor A^2)$.  This shows that in the case $P =
F_{\bld{1}}^{\cI}(A)$, we have an isomorphism
$F_{\bld{1}}^{\cI}(A)^{\boxtimes s} \iso F_{\bld{1}^{\concat
    s}}^{\cI}(A^{\tensor s})$ of $\Sigma_s$-equivariant objects where
$\Sigma_s$ acts on the target by permuting both the $\tensor$-powers
of $A$ and the index set of the sum. The commutative $\cI$-dga
$\bC(F_{\bld{1}}^{\cI}(A))$ will be of particular importance for us,
and we note that the above implies  
\begin{equation}\label{eq:free-comm-I-dga-explicit}
\bC(F_{\bld{1}}^{\cI}(A))(\bld{m}) \iso
 \textstyle \bigoplus_{s \geq 0} \left(\left(  \textstyle
     \bigoplus_{\cI(\bld{1}^{\concat s},\bld{m})} A^{\tensor
       s}\right)/\, \Sigma_s\right). 
\end{equation}

\subsection{Homotopy colimits of commutative
  \texorpdfstring{$\cI$}{I}-dgas} 
We will now construct an operad
action on the homotopy colimit of a commutative $\cI$-dga. Our
construction involves a symmetric monoidal structure on simplicial
chain complexes:
\begin{definition}
Let $A$ and $B$ be two simplicial chain complexes. Their tensor
product $A \hot B$ is the simplicial chain complex with 
\[\bigoplus_{\ell + m = n} A_{p,\ell} \otimes B_{p, m}\] 
in simplicial degree $p$ and chain degree $n$. The simplicial
structure maps act coordinatewise and the differential $d^{\hot}$ is 
\[ d^{\hot}(a \otimes b) = d(a) \otimes b + (-1)^\ell a \otimes d(b)\]
for $a \otimes b \in A_{p,\ell} \otimes B_{p, m}$. 
The symmetry isomorphism 
$c \colon A \hot B \ra B \hot A$ sends a homogeneous element $a
\otimes b$ as above to $(-1)^{\ell \cdot 
  m} b \otimes a$. 
\end{definition}

We denote by $\widetilde{\Sigma}_s$ the translation category of the
symmetric group $\Sigma_s$. Its objects are elements $\sigma \in
\Sigma_s$ and $\tau \in \Sigma_s$ is the unique morphism from $\sigma$
to $\tau \circ \sigma$ in $\widetilde{\Sigma}_s$. Since there is exactly
one morphism between each pair of objects, we get a functor
\begin{equation}\label{eq:be-in-cat} \widetilde{\Sigma}_s \times \widetilde{\Sigma}_{j_1}\times \dots \times  \widetilde{\Sigma}_{j_s} \to \widetilde{\Sigma}_{j_1+ \dots + j_s}
\end{equation}
by specifying that $(\sigma;\tau_1,\dots, \tau_s)$ is sent to
the composite \[ (\tau_{\sigma^{-1}(1)}\concat \dots \concat
  \tau_{\sigma^{-1}(s)}) \circ \sigma(j_1,\dots,j_s)\]
of the block permutation
$\sigma(j_1,\dots,j_s) \colon \bld{j_1}\concat \dots \concat \bld{j_s} \to \bld{j_{\sigma^{-1}(1)}} \concat \dots \concat \bld{j_{\sigma^{-1}(s)}}$ induced by $\sigma$ and the concatenation of the $\tau_{\sigma^{-1}(j)}$ (see~\cite[\S 4]{maypermcat} and \cite[Correction~34 on p. 490]{CLM_homology-iterated-loop-spaces}).

The action~\eqref{eq:be-in-cat} is associative, unital, and
symmetric. It turns the collection of categories
$(\widetilde{\Sigma}_n)_{n\geq 0}$ into an operad $\widetilde{\Sigma}$
in the category $\cat$ of small categories. For the next definition,
we use that the nerve functor $\cN\colon \cat \to \sset$ and the
$k$-linearization $k\{-\}\colon \sset \to \sMod_k$ are strong
symmetric monoidal and that the associated chain complex functor
$C_*\colon \sMod_k \to \chk$ is lax symmetric monoidal (compare
Proposition~\ref{prop:shuffle} below).
\begin{definition}\label{def:Barratt-Eccles-operad}
  The Barratt--Eccles operad is the $E_{\infty}$ operad $\cE$ in $\chk$
  with $\cE_n = C_*(k\{\cN(\widetilde{\Sigma}_n)\})$ and operad
  structure induced by the functor~\eqref{eq:be-in-cat}.
\end{definition}
The commutativity operad $\cC$ in $\chk$ is the operad with $\cC_n =
S^0$ concentrated in chain complex level $0$. The 
operad $\cE$ admits a canonical operad map $\cE \to \cC$ which is a
quasi-isomorphism in each level. Moreover, $\cE_n$ is a free
$k[\Sigma_n]$-module for each~$n$. Thus $\cE$ is an $E_{\infty}$
operad in $\chk$ in the terminology
of~\cite[Definition~4.1]{Mandell_cochain-multiplications}. 

Applying the nerve to $\widetilde{\Sigma}$ defines
an operad in $\sset$ that is more commonly referred to as the
Barratt--Eccles operad. It is well known that the latter operad acts
on the nerve of a permutative category \cite[Theorem
4.9]{maypermcat}. The next lemma recalls the underlying action of
$\widetilde{\Sigma}$ for the permutative category $\cI$.  

\begin{lemma}
The operad $\widetilde{\Sigma}$ in $\cat$
acts on $\cI$. On objects $\sigma$ in $\widetilde{\Sigma}_n$ and
$\bld{m_i}$ in $\cI$, the action is given by  
$(\sigma; \bld{m_1}, \ldots, \bld{m_n}) \mapsto
\bld{m_{\sigma^{-1}(1)}} \sqcup \ldots \sqcup
\bld{m_{\sigma^{-1}(n)}}.$\end{lemma}
\begin{proof}
  This is a special case of~\cite[Lemmas 4.3 and
  4.4]{maypermcat}. Functoriality in morphisms of 
  $\widetilde{\Sigma}_n$ uses the symmetry isomorphism of $\cI$ while
  the functoriality in $\cI$ is the evident one. 
\end{proof}
The next result is our main motivation for considering the Barratt--Eccles
operad. It is analogous to the result about $\cI$-diagrams in spaces 
established in~\cite[Proposition 6.5]{Schlichtkrull_Thom-symmetric}.
\begin{theorem}\label{thm:E-infinity-action-on-hocolim-I}
For every commutative $\cI$-dga $E$, the chain complex
$E_{h\cI}$ has a natural action of the Barratt--Eccles operad $\cE$.
\end{theorem}
\begin{proof}
  We can view the simplicial $k$-module
  $k\{\cN(\widetilde{\Sigma}_n)\}$ as a simplicial chain complex
  concentrated in chain degree $0$. The operad structure of
  $\widetilde{\Sigma}$ turns these simplicial $k$-modules into
  an operad in $\sMod_k$ and in $s\chk$. We construct an action
  \[ k\{\cN(\widetilde{\Sigma}_s)\} \hot \srep(E)^{\hot s}\to
    \srep(E). \]
  It is enough to specify the action of a 
  $q$-simplex $\sigma_0 \xleftarrow{\tau_1} \sigma_1 \leftarrow \dots \xleftarrow{\tau_q} \sigma_q$ in
  $\cN(\widetilde{\Sigma}_s)$ on a collection of elements
  $(\alpha^i_1, \dots, \alpha^i_q; x^i)$ in $\srep(E)[q]_{p_i}$ where
  $\alpha^i_j \colon \bld{n^i_j} \to \bld{n^i_{j-1}}$ is a map in
  $\cI$ and $x^i$ is an element in $E(\bld{n^i_q})_{p_i}$. On the
  indices $(\alpha^i_1, \dots, \alpha^i_q)$ for the sums in the
  simplicial replacement, we use the action of $(\tau_1,\dots,
  \tau_q)$ provided by the previous lemma. As element in $
  E(\bld{n_q^{\sigma_q^{-1}(1)}} \concat\dots \concat
  \bld{n_q^{\sigma_q^{-1}(s)}})_{p_1+\dots+ p_s}$ we take the product
  $x^{\sigma_q^{-1}(1)}\cdots x^{\sigma_q^{-1}(s)}$. Since $E$ is
  commutative, this does indeed define an operad action in $\schk$.
  By Propositions~\ref{prop:shuffle} and~\ref{prop:Totsymmon} below,
  the composite $\Tot C_*$ is lax symmetric monoidal. Hence it follows
  that $\cE$ acts on $E_{h\cI}$.
\end{proof}

\subsection{Monoidality of \texorpdfstring{$C_*$}{C} and \texorpdfstring{$\Tot$}{Tot}}
It remains to verify the monoidal properties of $C_*$ and $\Tot$ that 
were used in the proof of Theorem~\ref{thm:E-infinity-action-on-hocolim-I}.
\begin{definition}
Let $Y$ and $Z$ be two objects in $\chk(\chk)$. Their tensor
product is $Y \otimes Z$ is the object in $\chk(\chk)$ with  
\[ (Y \otimes Z)_{p,q}= \bigoplus_{a_1+a_2 = p}
\bigoplus_{b_1+b_2 = q} Y_{a_1,b_1} \otimes Z_{a_2,b_2}\] 
and differentials $d_h^\otimes(y \otimes z) = d_h(y) \otimes z +
(-1)^{a_1}y \otimes d_h(z)$ and $d_v^\otimes(y \otimes z) =
{d_v(y) \otimes z} + {(-1)^{b_1}y \otimes d_v(z)}$. The symmetry
isomorphism 
$\tau \colon Y \otimes Z \ra Z \otimes Y$ sends a homogeneous element
$y \otimes z \in Y_{a_1,b_1} \otimes Z_{a_2,b_2}$  to $(-1)^{a_1a_2 +b_1b_2} z
\otimes y$. 
\end{definition}

\begin{proposition} \label{prop:shuffle}
The functor $C_* \colon \schk \ra {\chk}(\chk)$ is
lax symmetric monoidal. 
\end{proposition}
\begin{proof}
  As in \cite[Theorem VIII.8.8]{MacLane-homology} we denote
  $(p,q)$-shuffles as two disjoint subsets $\mu_1 < \ldots < \mu_p$
  and $\nu_1 < \ldots < \nu_q$ of $\{0,\ldots, p+q-1\}$. For
  simplicial chain complexes $A$ and $B$ we define maps
\[ \sh_{A,B} \colon  C_*(A) \otimes C_*(B) \ra C_*(A \hot B)\]
that turn $C_*$ into a lax symmetric monoidal functor: If $a \otimes
b$ is a homogeneous element in $A_{r_1,r_2} \otimes 
B_{s_1,s_2}$ we set 
\[ \sh_{A,B}(a \otimes b) = \sum_{(\mu,\nu)} \text{sgn}(\mu,\nu)
s_{\nu_{s_1}} \circ \ldots \circ s_{\nu_1}(a) \otimes s_{\mu_{r_1}}
\circ \ldots \circ s_{\mu_1}(b).\]
Here, the sum runs over all $(r_1,s_1)$-shuffles $(\mu, \nu)$ and
$\text{sgn}(\mu,\nu)$ denotes the signum of the associated
permutation. 

As the simplicial structure maps of $A$ and $B$ commute with $d^A$ and
$d^B$, it follows that $\sh$ commutes with the vertical
differential. The proof that the horizontal differential is
compatible with $\sh$ is the same as for $\sh$ in the context of
simplicial modules. 

It remains to show that $\sh$ turns $C_*$ into a lax symmetric
monoidal functor, \ie, we have to show that 
\begin{equation} \label{eq:symmshuffle}
C_*(c) \circ \sh(a \otimes b) = \sh \circ \tau(a \otimes b)
\end{equation}
for any homogeneous element $a \otimes b \in A_{r_1,r_2} \otimes
B_{s_1,s_2}$.  As $\tau(a \otimes b) = (-1)^{r_1s_1+r_2s_2} b \otimes
a$, the right-hand side of equation \eqref{eq:symmshuffle}
is 
\[\sum_{(\xi,\zeta)} (-1)^{r_1s_1+r_2s_2} \text{sgn}(\xi,\zeta) 
s_{\zeta_{s_1}} \circ \ldots \circ s_{\zeta_1}(b) \otimes s_{\xi_{r_1}}
\circ \ldots \circ s_{\xi_1}(a)\]
with $(\xi,\zeta)$ being $(s_1,r_1)$-shuffles, 
whereas the left-hand side of the equation gives 
\[(-1)^{r_2s_2} \sum_{(\mu,\nu)} \text{sgn}(\mu,\nu)
 s_{\mu_{r_1}} \circ \ldots \circ s_{\mu_1}(b) \otimes s_{\nu_{s_1}}
 \circ \ldots \circ s_{\nu_1}(a) \]
because $\tau$ introduces the sign $(-1)^{r_2s_2}$. Precomposing with
the permutation that exchanges the blocks $0 < \ldots < r_1-1$ and $r_1
 < \ldots < r_1+s_1 -1$ gives a bijection between the summation
 indices and introduces the sign $(-1)^{r_1s_1}$. Hence the two sides
agree. 
\end{proof}

\begin{proposition} \label{prop:Totsymmon}
The functor $\Tot$ is strong symmetric monoidal. 
\end{proposition}
\begin{proof}
Spelling out what $\Tot(Y) \otimes \Tot(Z)$ is in degree $n$ we obtain 
\[(\Tot(Y) \otimes \Tot(Z))_n \cong  \bigoplus_{r_1+r_2+s_1+s_2=n}
Y_{r_1,r_2} \otimes Z_{s_1,s_2},\] and we send a homogeneous element
$y \otimes z \in  Y_{r_1,r_2} \otimes Z_{s_1,s_2}$ to the element 
\[ (-1)^{r_2s_1} y \otimes z \in \Tot(Y \otimes Z)_n \cong
\bigoplus_{r_1+s_1+r_2+s_2=n} Y_{r_1,r_2} \otimes Z_{s_1,s_2}.\]
This gives isomorphisms 
\[\varphi_{Y,Z} \colon \Tot(Y) \otimes \Tot(Z) \ra \Tot(Y \otimes Z)\]
that are associative.  It is clear that $\Tot$ respects the unit up to 
isomorphism. 

The maps $\varphi_{Y,Z}$ are compatible with the differential: Let $y
\otimes z$ be a homogeneous element in $Y_{r_1,r_2} \otimes
Z_{s_1,s_2}$. The composition $d_{\Tot} \circ \varphi$ applied to $y 
\otimes z$ gives 
\begin{align*}
d_{\Tot} \circ \varphi(y \otimes z) & = (-1)^{r_2s_1}d_h^\otimes(y
\otimes z) + (-1)^{r_2s_1}(-1)^{r_1+s_1}d_v^\otimes(y \otimes z) \\
& = (-1)^{r_2s_1}d_h(y) \otimes z + (-1)^{r_2s_1+r_1} y \otimes d_h(z) \\
& \phantom{=} + (-1)^{r_2s_1+r_1+s_1}d_v(y) \otimes z + (-1)^{r_2s_1+r_1+s_1+r_2} y
\otimes d_v(z).
\end{align*}
First applying the differential to $y \otimes z$ and then $\varphi$
yields 
\begin{align*}
 & \varphi(d_{\Tot}(y) \otimes z + (-1)^{r_1+r_2}y \otimes
 d_{\Tot}(z)) \\ 
= & \varphi(d_h(y) \otimes z + (-1)^{r_1}d_v(y) \otimes z +
(-1)^{r_1+r_2}y \otimes d_h(z) + (-1)^{r_1+r_2+s_1} y \otimes d_v(z)) \\
= & (-1)^{r_2s_1}d_h(y) \otimes z + (-1)^{r_1+(r_2-1)s_1}d_v(y) \otimes z +
(-1)^{r_1+r_2+r_2(s_1-1)}y \otimes d_h(z)\\
& + (-1)^{r_1+r_2+s_1+r_2s_1} y \otimes d_v(z)
\end{align*}
thus both terms agree. 

We denote the symmetry isomorphism in the category of chain complexes
by $\chi$. Then 
\[\varphi \circ \chi(e\otimes f) = \varphi((-1)^{(r_1+r_2)(s_1+s_2)} f
  \otimes e) = (-1)^{r_1s_1+r_2s_2+s_1r_2 +2s_2r_1}f \otimes e\]
and this is equal to 
\[\Tot(\tau) \circ \varphi (e
\otimes f) = \Tot(\tau)((-1)^{r_2s_1}e \otimes f) = (-1)^{r_2s_1+r_1s_1+r_2s_2} f
\otimes e.\qedhere\] 
\end{proof}
\begin{remark}
One can also consider a symmetric monoidal structure on $\chk(\chk)$ with
the same underlying tensor product but with symmetry
isomorphism 
\[ y \otimes z \mapsto (-1)^{(r_1+r_2)(s_1+s_2)} z \otimes y\]
for homogeneous elements $y \otimes z \in Y_{r_1,r_2} \otimes
Z_{s_1,s_2}$. Then one can take $\varphi$ in
Proposition~\ref{prop:Totsymmon} to be the identity. However, this
symmetry isomorphism is \emph{not} compatible with the shuffle transformation
from the proof of Proposition \ref{prop:shuffle}. 
\end{remark}
\begin{remark}
  For a simplicial chain complex $A$ one can also consider a
  normalized object $N(A) \in \chk(\chk)$ where one divides out by the
  subobject generated by degenerate elements. As the simplicial
  structure maps commute with the differential of $A$, this is
  well-defined, and the proof of Proposition \ref{prop:shuffle} can be
  adapted as in \cite[Corollary VIII.8.9]{MacLane-homology} to show
  that the functor $N\colon \schk \ra \chk(\chk)$ is also lax
  symmetric monoidal. Consequently, one can also use $N$ instead of
  $C_*$ in the definition of the Barratt--Eccles operad $\cE$ and the
  homotopy colimit $P_{h\cI}$ so that
  Theorem~\ref{thm:E-infinity-action-on-hocolim-I} remains valid. 
\end{remark}
\section{Cochain functors with values in
  \texorpdfstring{$\cI$}{I}-chain complexes}
In this section we construct the functor $A^{\cI}$ discussed in the introduction
and a version of the ordinary cochains with values in $\cI$-chain complexes. 
\subsection{Adjunctions induced by simplicial objects}
We briefly recall an ubiquitous construction principle for adjunctions 
that we will later apply to simplicial objects in the categories
of commutative $\cI$-dgas and $\cI$-chain complexes in order to define
the commutative $\cI$-dga of polynomial forms on a simplicial set and
an integral version of the Sullivan realization functor (see Definition
\ref{def:aiandki}). 

\begin{construction}\label{constr:adj}
  Let $D_{\bullet} \colon \Delta^{\op} \to \cD$ be a simplicial object
  in a complete category $\cD$. Passing to opposite categories,
  $D_{\bullet}$ gives rise to a functor
  $\widetilde{D}^{\bullet} \colon \Delta \to \cD^{\op}$. Since $\cD$
  is complete, $\cD^{\op}$ is cocomplete. Hence restriction and left Kan
  extension along $\Delta \to \sset, [p]\mapsto \Delta^p$ define an
  adjunction
  \[ \widetilde{D}\colon \sset \rightleftarrows \cD^{\op} \colon K_D.\]
  Writing $D \colon \sset^{\op} \to \cD$ for the opposite of
  $\widetilde{D}$, this implies that for a simplicial set $X$ and an
  object $E$ of $\cD$, we have a natural isomorphism
  \begin{equation}\label{eq:D-adjunction} \cD(E, D(X)) =
    \cD^\op(\widetilde{D}(X),E) \iso \sset(X, K_D(E))
  \end{equation}
  exhibiting $D$ and $K_D$ as \emph{contravariant right adjoint}
  functors. Unraveling definitions, the contravariant functors $K_D$
  and $D$ are given by $K_D(E)_\bullet = \cD(E,D_\bullet)$ and $D(X) =
  \lim_{\Delta^p \to X} D_p$ where the limit is taken over the
  category of elements of $X$. In the special case $\cD =
  \mathrm{Set}$, writing $X$ as a colimit of representable functors
  indexed over its category of elements provides a natural bijection
  $D(X) \iso \sset(X,D)$.

  The functor $D$ extends the original
  functor $D_{\bullet}$ in that there is a natural isomorphism
  $D_{\bullet} \iso D(\Delta^{\bullet})$. The construction is also
  functorial in $D_{\bullet}$, \ie, a natural transformation
  $D_{\bullet} \to D'_{\bullet}$ of functors $\Delta^{\op}\to \cD$
  induces a natural transformation $D \to D'$ of functors
  $\sset^{\op} \to \cD$.
\end{construction}
We note an immediate consequence of having the adjunction $(\widetilde
D, K_D)$.  
\begin{lemma}\label{lem:D-colimits-limits}
  The functor $D$ takes colimits in $\sset$ to limits in $\cD$, 
and $K_D$ takes colimits in  $\cD$ to limits in $\sset$.\qed
\end{lemma}
When $D_{\bullet} \colon \Delta^{\op} \to \CchIk$ is a simplicial
object in commutative $\cI$-dgas, we may apply
Construction~\ref{constr:adj} both to $D_{\bullet}$ and to its
composite $D'_{\bullet} = UD_{\bullet}$ with the forgetful functor
$U\colon \CchIk \to \chIk$. Since the extensions of $D_{\bullet}$ and
$D'_{\bullet}$ to functors on $\sset$ are defined by limit
constructions and $U$ commutes with limits, we have a natural
isomorphism $U(D(X)) \iso D'(X)$ for a simplicial set $X$. The
adjoints $K_{D}$ and $K_{D'}$ are related by a natural isomorphism
$K_{D'} \iso K_{D}\circ \bC \colon (\chIk)^{\op} \to \sset$.  An
analogous remark applies to simplicial objects of algebras in $\chIk$
over a more general operad than the commutativity operad. 

For $D_{\bullet} \colon \Delta^{\op} \to \CchIk$, the fact that
$\CchIk \to \mathrm{Set}$, $E\mapsto E(\bld{m})_q$ commutes with
limits implies that the underlying set of $D(X)(\bld{m})_q$ is
$\sset(X,D_{\bullet}(\bld{m})_q)$. The pointwise $k$-module structure,
differentials and multiplications on these sets give rise to the
commutative $\cI$-dga structure on $D(X)$. 

\subsection{The commutative \texorpdfstring{$\cI$}{I}-dga version of
  polynomial forms}\label{subsec:I-version-polynomial} Composing the left adjoints in the adjunctions
$(F_{\bld{1}}^{\cI},\mathrm{Ev_{\bld{1}}})$ and $(\bC,U)$ introduced
in~\eqref{eq:free-I-chain-complex} and~\eqref{eq:free-commutative}
provides a left adjoint $\bC F^{\cI}_{\bld{1}} \colon \chk \to \CchIk$
made explicit in~\eqref{eq:free-comm-I-dga-explicit}.  We are
particularly interested in the commutative $\cI$-dga
$\bC F^{\cI}_{\bld{1}}(D^0)$.  For an element $i \in k$, the
$k$-module map $(D^0)_0 = k \to k  = \mathrm{Ev}_1(U^{\cI})_0$ determined by
$1\mapsto i$ gives rise to a map
$\varepsilon_i \colon \bC F^{\cI}_{\bld{1}}(D^0) \to U^{\cI}$. We
write $U^{\cI}_0$ and $U^{\cI}_1$ for the two commutative
$\bC F^{\cI}_{\bld{1}}(D^0)$-algebras resulting from the elements
$0,1 \in k$.

\begin{definition}\label{def:two-sided-bar-CchIk}
  We let $A^{\cI}_{\bullet}\colon \Delta^{\op} \to \CchIk$ be the
  simplicial commutative $\cI$-dga given by the two-sided bar
  construction
  \begin{equation}\label{eq:two-sided-bar-CchIk} [p] \mapsto A^{\cI}_p
    = B_p(U^{\cI}_0,\bC F^{\cI}_{\bld{1}}(D^0) ,U^{\cI}_1) = U^{\cI}_0
    \boxtimes \bC F^{\cI}_{\bld{1}}(D^0)^{\boxtimes p} \boxtimes
    U^{\cI}_1\ .
\end{equation}
As with the space level version (see e.g.~\cite{may-gils}), the outer
face maps are provided by the module structures of $U^{\cI}_0$ and
$U^{\cI}_1$ resulting from the above algebra structures, the inner
face maps come from the multiplication of
$\bC F^{\cI}_{\bld{1}}(D^0)$, and the degeneracy maps are induced by
its unit. \end{definition}

To make this simplicial object more explicit, we write $D^0_r$ for the
chain complex with copies of $k$ on generators $r$ in degree $0$ and
on $dr$ in degree $-1$ and $0$ elsewhere. Its non-zero differential is
$d(a\cdot r) = a\cdot dr$.  Since $ \bC F^{\cI}_{\bld{1}}$ is left
adjoint and $U^{\cI}$ is the unit for $\boxtimes$, commuting
$\bC F_1^{\cI}$ with coproducts provides an isomorphism of commutative
$\cI$-dgas
\[ A^{\cI}_p \iso \bC F^{\cI}_{\bld{1}}(D^0_{r_1(p)} \oplus \dots \oplus
D^0_{r_p(p)}) \]
where the generators $r_1(p), \dots, r_p(p)$ correspond to the $p$
copies of $\bC F^{\cI}_{\bld{1}}(D^0)$.  By adjunction, maps
$f \colon \bC F^{\cI}_{\bld{1}}(D^0_{r_1(p)} \oplus \dots \oplus
D^0_{r_p(p)}) \to E$
in $\CchIk$ correspond to families of elements
$f(r_1(p)), \dots, f(r_p(p)) \in E(\mathbf{1})_0$.  

We now set $r_0(p) = 0$ and define $r_{p+1}(p)$ to be the image of $1$
under the map
\[k = U^{\cI}(\bld{1})_0 \to \bC F^{\cI}_{\bld{1}}(D^0_{r_1(p)} \oplus
\dots \oplus D^0_{r_p(p)})(\bld{1})_0\]
induced by the unit.  With this notation, the simplicial structure
maps of the two sided bar construction~\eqref{eq:two-sided-bar-CchIk}
are determined by requiring
\begin{equation}\label{eq:simp-structure-AI}
d_i(r_j(p)) = \begin{cases} r_j(p-1) & \text{if } j \leq i\\
  r_{j-1}(p-1) & \text{if } j> i,  
\end{cases}
\quad   s_i(r_j(p)) = \begin{cases} r_j(p+1) & \text{if } j \leq i  \\
  r_{j+1}(p+1) & \text{if } j > i. 
\end{cases}
\end{equation}

Applying Construction \ref{constr:adj}, we obtain the following pair
of adjoint functors. 
\begin{definition} \phantomsection \label{def:aiandki}
  \begin{enumerate}[(i)]
    \item
 The \emph{commutative $\cI$-dga of polynomial
forms on a simplicial set $X$}, $A^{\cI}(X)$, is defined as
\[ A^{\cI}(X) = \sset(X, A^{\cI}_\bullet).\]
This defines a functor $A^{\cI} \colon \sset^{\op} \ra \CchIk$. 
\item
  Its adjoint functor $\langle -\rangle_\cI \colon \CchIk^{\op}\ra \sset$ sends
  a commutative $\cI$-dga
$E$ to
\[ \langle E\rangle_\cI = \CchIk(E, A^{\cI}_\bullet). \]
The simplicial set $\langle E\rangle_\cI$ is the
\emph{Sullivan realization of $E$}.
\end{enumerate}
\end{definition}

For a simplicial $k$-module $Z \colon \Delta^{\op} \to \Mod_k$,
\emph{extra degeneracies} are a family of $k$-linear maps
$s_{p+1}\colon Z_p \to Z_{p+1}$ satisfying $d_{p+1} s_{p+1} =
\id_{Z_p}$ if $p \geq 0$, $d_i s_{p+1} = s_p d_i \colon Z_p \to Z_p$
if $p\geq 1$ and $0\leq i \leq p$, and $0 = d_0 s_1 \colon Z_0 \to
Z_0$. The presence of extra degeneracies implies that $Z$ is
contractible to $0$ (in the sense that $Z \to 0$ is a weak equivalence
in $\sMod_k$) since the maps $(-1)^{p+1} s_{p+1}$ define a contracting
homotopy for the chain complex $C_*(Z)$.  

The following lemma is the technical backbone for our homotopical analysis of
the prolongation $A^{\cI}$ of $A^{\cI}_{\bullet}$ in Section~5. It
is analogous to~\cite[Proposition 1.1]{Bousfield-G_rational}.  
\begin{lemma}\label{lem:AI-contractible}
  Let $\bld{m}$ be an object of $\cI$ with $m = |\bld{m}|\geq 1$. Then for
  all integers $q$ satisfying $0 \geq q > -m$, the simplicial
  $k$-module $A^{\cI}_\bullet(\bld{m})_q$ is contractible to $0$.
\end{lemma}

\begin{remark}\label{rem:AI-not-contractible-in-I-degree-0}
  The statement of the lemma does not hold for general $\bld{m}$ and
  $q$. The easiest case is $\bld{m} = \bld{0}$ and $q=0$ where
  $A^{\cI}_\bullet(\bld{0})_0$ is the constant simplicial object on
  $k$, which is not contractible. One can also show that
  $\pi_1(A^{\cI}_\bullet(\bld{1})_{-1})$ is non-trivial.
\end{remark}

\begin{proof}[Proof of Lemma~\ref{lem:AI-contractible}]
  Let $A$ be a chain complex. The canonical bijection
\[ \cI(\bld{1}^{\concat s}, \bld{m})/\Sigma_s \to \{ T\subseteq
  \bld{m}\;|\; |T|=s\}, \qquad [\alpha] \mapsto \mathrm{im}(\alpha)\]
induces natural isomorphisms
\begin{equation}\label{eq:F1IAm-identification}\bC(F_{\bld{1}}^{\cI}(A))(\bld{m}) = \bigoplus_{s\geq 0}\left(\bigoplus_{\cI(\bld{1}^{\concat s},\bld{m})} A^{\otimes s}/\Sigma_s\right) \xrightarrow{\iso} \bigoplus_{T \subseteq \bld{m}} A^{\otimes T}\xrightarrow{\iso}  (A \oplus S^0)^{\tensor m}.
\end{equation}
Here the last isomorphism sends the tensor power indexed by $T$ to an
iterated tensor product of copies of $A$ and $S^0$ with copies of $A$
placed at the entries indexed by $T$. 

The isomorphism~\eqref{eq:F1IAm-identification} specializes to an isomorphism
\[  A^{\cI}_p(\bld{m})_q \iso \left( \left(D^0_{r_1(p)} \oplus
      \dots \oplus D^0_{r_p(p)} \oplus S^0_{r_{p+1}(p)}\right)^{\otimes \bld{m}}\right)_q\]
where we now write $r_{p+1}(p)$ for the generator of $(S^0)_0$. Under this identification, the simplicial structure maps of $[p] \mapsto A^{\cI}_p(\bld{m})_q$ are again determined by~\eqref{eq:simp-structure-AI}. 

As a first step, we now notice that $A^{\cI}_{\bullet} (\bld{1})_0$ is contractible since $s_{p+1}(r_j(p)) = r_{j}(p+1)$ defines extra degeneracies for this simplicial object. For the case of a general $\bld{m}$ and $0\geq q > -m$, we notice that the above isomorphisms induce an isomorphism of simplicial objects  \[ A^{\cI}_p(\bld{m})_q \iso \bigoplus_{\substack{q_1+ \dots + q_m  = q\\ q_i = 0, -1} }  A^{\cI}_p(\bld{1})_{q_1} \otimes \dots \otimes  A^{\cI}_p(\bld{1})_{q_m}.  
\] 
Since $q > -m$, each summand has at least one tensor factor that is of chain complex degree $0$ and thus contractible by the previous step. Since the shuffle map is a chain homotopy equivalence~\cite[Theorem VIII.8.1]{MacLane-homology}, it follows that each summand and thus the whole sum is contractible. 
\end{proof}

\subsection{Ordinary cochains} 
Let $C(X;k)$ be the cochains with values in $k$ on the simplicial set
$X$, viewed as a homologically graded chain complex concentrated in
non-positive degrees. (At this point, we disregard its cup product
structure.) So for $q \geq 0$, we have
$C(X;k)_{-q} = \mathrm{Set}(X_q,k)$ with the pointwise $k$-module
structure and differential induced by the face maps of $X$. The
cochains on the standard $n$-simplices assemble to a functor
$C_{\bullet} \colon \Delta^{\op} \to \chk, [p] \mapsto
C(\Delta^{p};k)$.
The following lemma is well known (see e.g.~\cite[Lemma 10.11 and Lemma
10.12(ii)]{FHT}). 

\begin{lemma}\phantomsection \label{lem:properties-usual-cochains}
  \begin{enumerate}[(i)]
  \item The extension of $C_{\bullet}$ to a functor $\sset^{\op}\to \chk$ resulting 
from Construction~\ref{constr:adj} is naturally isomorphic to $C(-;k)$. 
\item  For all $q \in \mathbb Z$, the simplicial $k$-module
  $C_{\bullet,q} = C(\Delta^{\bullet};k)_q$ is contractible to $0$.  
\end{enumerate}
\end{lemma}
\begin{proof}
  For (i), we note that the description of the extension as
  $ \lim_{\Delta^p \to X} C(\Delta^p;k)$ implies that there is a
  natural map from $C(X;k)$. Writing $X$ as a colimit of representable
  functors over its category of elements, the evaluation of this map
  at $q$ is a bijection since taking maps into $k$ turns colimits into limits.

  For (ii), we only need to consider the case $q \leq 0$, set $n = -q$
  and define \[s_{p+1}\colon C(\Delta^p;k)_q \to C(\Delta^{p+1};k)_q\]
  on $f \colon (\Delta^p)_n \to k$ as follows: We set $s_{p+1}(f)
  \colon (\Delta^{p+1})_n \to k$ to be $0$ on all $n$-simplices not in
  the image of $d^{p+1}\colon \Delta^p \to \Delta^{p+1}$ and require
  that $s_{p+1}(f)$ restricts to $f$ on the last face.
  Identifying $\Delta^{p+1}_n$ with
$\Delta([n],[p+1])$, this means 
  that $s_{p+1}(f)(d^{p+1}\alpha') = f(\alpha')$ and
  $s_{p+1}(f)(\alpha) = 0$ if $p+1 \in \alpha([n])$. Then for
  $\beta\colon [n] \to [p]$ , the equation $d_{p+1}(s_{p+1}(f))(\beta)
  = \beta$ holds by definition, and $d_0s_1 = 0$ in simplicial degree
  $0$ is also immediate.  Now assume $p\geq 1$. If $\beta$ has $p$ in its
  image, then $d_is_{p+1}(f)(\beta) = 0 =
  s_pd_i(f)(\beta)$. Otherwise, we must have $\beta = d^p \beta'$ and
  thus
\begin{align*} d_i s_{p+1}(f)(d^p \beta') & = s_{p+1}(f)(d^id^p
  \beta') = s_{p+1}(f)(d^{p+1}d^i \beta') \\
& = f(d^i \beta') = (d_if)(\beta') = s_p d_i(f)(d^p\beta'). \qedhere
\end{align*}
\end{proof}
For later use, we lift $C_{\bullet}$ to $\cI$-chain complexes by defining 
\[C^{\cI}_{\bullet}\colon \Delta^{\op} \to \chIk,\quad [p] \mapsto
F_{\bld{0}}^{\cI}(C(\Delta^{p};k)).\]

\begin{corollary}\phantomsection \label{cor:properties-of-CI}
\begin{enumerate}[(i)]
\item The extension $C^{\cI}$ of $C^{\cI}_{\bullet}$ to a functor
  $\sset^{\op} \to \chIk$ resulting from Construction~\ref{constr:adj} is
  naturally isomorphic to $X \mapsto F^{\cI}_{\bld{0}}C(X; k)$.
\item For all $q \in \mathbb Z$ and $\bld{m}$ in $\cI$, the
  simplicial $k$-module
  $C^{\cI}_\bullet(\bld{m})_q =
  F_{\bld{0}}^{\cI}(C(\Delta^{\bullet};k)_q)$
  is contractible to $0$. \qed
\end{enumerate}
\end{corollary}

\section{Homotopy theory of \texorpdfstring{$\cI$}{I}-chain complexes  and commutative \texorpdfstring{$\cI$}{I}-dgas}\label{sec:model_str}
In this section we review and set up results about model category
structures on $\cI$-chain complexes and commutative $\cI$-dgas. Much
of this is motivated by (and analogous to) the corresponding results
for space valued functors developed in~\cite[\S 3]{Sagave-S_diagram}.
 
We continue to consider the category of unbounded chain complexes
$\chk$ and equip it with the projective model structure whose weak
equivalences are the quasi-isomorphisms and whose fibrations are the
level-wise surjections~\cite[Theorem 2.3.11]{Hovey_model}. It has the
inclusions $S^{q-1}\hookrightarrow D^q$ as generating cofibrations and
the maps $0 \to D^q$ as generating acyclic cofibrations. We will also
need the following variant of this model structure.

\begin{proposition}
  Let $s$ be an integer. Then $\chk$ admits an \emph{$s$-truncated
    model structure} where a map $f \colon A \to B$ is a weak
  equivalence if $H_q(f)\colon H_q(A) \to H_q(B)$ is an isomorphism
  for all $q \geq s$ and a fibration if $f_q\colon A_q \to B_q$ is an
  epimorphism for all $q > s$. The $s$-truncated model structure is
  combinatorial and right proper.
\end{proposition}
\begin{proof}
  By shifting it is enough to consider the case $s =0$. The smart truncation 
  \[ \tau_{\geq 0} \colon \chk \to \mathrm{Ch}_k^{\geq 0}, \qquad A
    \mapsto (\dots \to A_2 \to A_1 \to \mathrm{ker}\, d_0^A) \] to
  non-negatively 
  graded chain complexes is right adjoint to the functor that adds
  copies of $0$ in negative degrees. The desired model structure
  arises by applying \cite[Theorem
  11.3.2]{Hirschhorn_model} to this adjunction and the standard
  projective model structure on $\mathrm{Ch}_k^{\geq 0}$~\cite[\S
  7]{Dwyer-S_model}. The assumptions of the theorem are trivially
  satisfied. The resulting model structure is combinatorial since
  $\chk$ is, and right proper because all objects are fibrant.
\end{proof}
\begin{remark}
  Since the usual long exact sequence argument is not applicable, we do
  not know if the $s$-truncated model structure is left proper. We do
  not investigate this further since it is not relevant for our
  applications.
\end{remark}

\subsection{Level model structures}
 We call an object $\bld{m}$ of $\cI$ \emph{positive} if
$|\bld{m}|\geq 1$ and write $\cI_+$ for the full subcategory of
positive objects in $\cI$. To ease notation, we write $m$ for
the cardinality of $\bld{m} = \{1,\dots,m\}$. 

A map $f\colon P \to Q$ in $\chIk$ is an \emph{absolute level
  equivalence} (resp. \emph{absolute level fibration}) if $f(\bld{m})$
is a quasi-isomorphism (resp. a fibration) in $\chk$ for all $\bld{m}$
in $\cI$.  A map $f\colon P \to Q$ in $\chIk$ is a \emph{descending
  level equivalence} (resp. \emph{descending level fibration}) if for
all $\bld{m}$ in $\cI_+$, the map $f(\bld{m})$ is a weak equivalence
(resp. fibration) in the $-m$-truncated model structure on
$\chk$.

\begin{proposition}
  These maps define an \emph{absolute level} and a \emph{descending
    level model structure} on $\chIk$. Both model structures are
  combinatorial and right proper, and the absolute level model
  structure is in addition left proper.
\end{proposition}
\begin{proof}
  For integers $s \leq t$, the identity functor is a left Quillen
  functor from the $t$-truncated model structure to the $s$-truncated
  model structure since $s$-truncated weak equivalences
  (resp. fibrations) are $t$-truncated weak equivalences
  (resp. fibrations).  To obtain the descending level model structure,
  we can therefore apply~\cite[Proposition 3.10]{HSS-retractive} to
  the constant functor $\mathcal C\colon \cI \to \mathrm{Cat}$ with
  value $\chk$ where $\mathcal C(\bld{m})$ is equipped with the
  $-m$-truncated model structure. The absolute level model structure
  arises by considering the same functor where $\mathcal C(\bld{m})$
  carries the standard projective model structure for all $\bld{m}$.
\end{proof}
The cofibrations in the absolute level model structure are the
retracts of relative cell complexes built out of cells of the form
$F^{\cI}_{\bld{m}}(S^{q-1} \hookrightarrow D^q)$ with $\bld{m}$ in
$\cI$ and $q \in \bZ$.  Here $F^{\cI}_{\bld{m}}$ is the free functor
defined in~\eqref{eq:free-I-chain-complex}. For the descending level
model structure, it follows from the proof of the previous proposition
that one may use the following set as generating cofibrations:
\[
\{ F^{\cI}_{\bld{m}}(S^{q-1} \hookrightarrow D^q) \;|\; \bld{m}\in \cI_+, q > -m\} \cup \{ F^{\cI}_{\bld{m}}(0\hookrightarrow D^{-m})  \;|\; \bld{m}\in \cI_+\} 
\]

\subsection{\texorpdfstring{$\cI$}{I}-model structures}
We now again use the homotopy colimit $P_{h\cI}$ from
Construction~\ref{constr:hocolim_I}.  A map $P \to Q$ in $\chIk$ is an
\emph{$\cI$-equivalence} if it induces a quasi-isomorphism
$P_{h\cI} \to Q_{h\cI}$. An $\cI$-chain complex $P$ is
\emph{absolute $\cI$-fibrant} if
$\alpha_*\colon P(\bld{m}) \to P(\bld{n})$ is a quasi-isomorphism for
all $\alpha\colon \bld{m} \to \bld{n}$ in $\cI$. It is
\emph{descending $\cI$-fibrant} if for all
$\alpha\colon \bld{m} \to \bld{n}$ in $\cI_+$ , the map
$\alpha_*\colon P(\bld{m}) \to P(\bld{n})$ is a weak equivalence in
the $-m$-truncated model structure, that is, if for all
$q\geq -m$, the map
$H_q(\alpha_*)\colon H_q(P(\bld{m})) \to H_q(P(\bld{n}))$ is an
isomorphism.

The next lemma will be needed to identify the $\cI$-equivalences as
part of a descending model structure.
\begin{lemma}\label{lem:I-equiv-descending-level-equiv}
  A map $P \to Q$ between descending $\cI$-fibrant objects in  $\chIk$
  is an $\cI$-equi\-va\-lence if and only if it is a descending level
  equivalence.  
\end{lemma}
\begin{proof}
For each positive $\bld{m}$ we consider the following commutative diagram: 
\[ 
  \xymatrix@-1pc{P(\bld{m}) \ar[rr] \ar[d] && \hocolim_{\cI_{\geq m}} P \ar[d]  \ar[rr]^-{\sim} && P_{h\cI}  \ar[d]\\ 
Q(\bld{m}) \ar[rr] && \hocolim_{\cI_{\geq m}} Q  \ar[rr]^-{\sim} && Q_{h\cI}  }
\] 
The right hand horizontal maps are induced by the inclusion
$\cI_{\geq m}$ of the full subcategory of objects of cardinality at
least $m$. Since this inclusion is homotopy cofinal
(compare~\cite[Corollary 5.9]{Sagave-S_diagram}), they are
quasi-isomorphisms. The left hand horizontal maps are induced by the
inclusion of the object $\bld{m}$ in $\cI_{\geq m}$. They are
$-m$-truncated equivalences by~\cite[Proposition
5.4]{Dugger_replacing} because the restrictions of $P$ and $Q$ to
$\cI_{\geq m}$ are diagrams of $-m$-truncated equivalences. The claim
then follows by two-out-of-three and the fact that a map is a
quasi-isomorphism if and only if it is an $s$-truncated equivalence
for all $s<0$.
\end{proof} 

Using this lemma, we can build the desired model structures:
\begin{proposition}\label{prop:I-model-structures}
  The absolute (resp. descending) level model structure on $\chIk$
  admits a left Bousfield localization with fibrant objects the
  absolute (resp. descending) $\cI$-fibrant objects. The weak
  equivalences in these two model structures coincide, and they are
  given by the $\cI$-equivalences. Both model structures are left
  proper and combinatorial. 
\end{proposition}
\begin{proof}
  For the absolute case, we apply~\cite[Theorem
  5.2]{Dugger_replacing}. Since the weak equivalences in the latter
  case are the maps that induce weak equivalences on the corrected
  homotopy colimits, they coincide with the $\cI$-equivalences.  The
  resulting model structure is left proper and 
  combinatorial since these properties are preserved by left Bousfield
  localization.

  Let $\mathcal N$ be the subcategory of $\cI$ given by the subset
  inclusions. Since the homotopy colimit over $\mathcal N$ is
  equivalent to the colimit of a suitable cofibrant replacement of the
  underlying $\mathcal N$-diagram, there is a natural isomorphism
  $\colim_{\bld{m}\in\mathcal N} H_q(P(\bld{m})) \iso
  H_q(\hocolim_{\bld{m} \in \mathcal N}P(\bld{m}))$. Hence any
  descending level equivalence $P \to Q$ induces a quasi-isomorphism
  $\hocolim_{\mathcal N}P \to \hocolim_{\mathcal N}Q$ and thus an
  $\mathcal I$-equivalence by \cite[Proposition
  2.2.9]{shipley_THH}. So the weak equivalences of the descending
  level model structure are contained in those of the absolute
  $\cI$-model structure, and the same is by definition true for the
  cofibrations. We can thus apply \cite[Theorem 2.1]{Cole_mixing} to
  obtain a model structure that has the weak equivalences and
  cofibrations of the second model structure to be constructed.  Using
  Lemma~\ref{lem:I-equiv-descending-level-equiv}, an object $P$ is
  descending $\cI$-fibrant if and only if every absolute $\cI$-fibrant
  replacement $P \to P'$ is a descending level equivalence. Hence the
  dual of \cite[Proposition 3.6]{Cole_mixing} shows that the fibrant
  objects of the second model structure are the descending
  $\cI$-fibrant objects. The model structure is combinatorial
  by~\cite[Lemma 4.6 and Proposition 2.2]{Barwick_left} and left
  proper since the absolute $\cI$-model structure is.
\end{proof}
We call these two model structures the \emph{absolute} and
\emph{descending} \emph{$\cI$-model structures} on $\chIk$ and their
fibrations \emph{absolute} and \emph{descending $\cI$-fibrations}.

\begin{remark}
  Analogous to~\cite[Proposition 6.16]{Sagave-S_diagram}, $\cI$-model
  structures on chain complexes can be constructed without relying on
  an abstract existence theorem for left Bousfield localizations. This
  has been done by Joachimi~\cite{Joachimi_Diplom} in the absolute
  case, and we expect that similar arguments apply in the descending
  case. The advantage of the direct approach is that it provides an
  explicit characterization of the $\cI$-fibrations by a homotopy
  pullback condition like in~\cite[Proposition~3.2]{Sagave-S_diagram}.
\end{remark}

\begin{corollary}\label{cor:equivalence-fibrations-in-localization}
  A map between descending $\cI$-fibrant objects is a descending level
  fibration if and only if it is a descending $\cI$-fibration.
\end{corollary}
\begin{proof}
  This follows from~\cite[Proposition 3.3.16]{Hirschhorn_model}.
\end{proof}

The following useful consequence of
Lemma~\ref{lem:I-equiv-descending-level-equiv} 
was already used in  the proof of Proposition~\ref{prop:I-model-structures}:
\begin{corollary}\label{cor:fibrant-repl-of-descending-I-fib}
  If $P$ is descending $\cI$-fibrant in $\chIk$, then a fibrant replacement
  $P \to P'$ in the absolute $\cI$-model structure is a descending
  level equivalence. \qed
\end{corollary}

We also note that the homology groups of $P_{h\cI}$ can be read off
from a fibrant object $P$ in the following way.  
\begin{lemma}\label{lem:hocolim-of-I-fibrant}
  If $P$ is absolute $\cI$-fibrant in $\chIk$ and $\bld{m}$ is any
  object in $\cI$, then the canonical map $P(\bld{m}) \to P_{h\cI}$ is
  a quasi-isomorphism. If $P$ is only descending $\cI$-fibrant, then
  the induced map $H_q(P(\bld{m})) \to H_q(P_{h\cI})$ is an
  isomorphism when $\bld{m}$ is positive and $q\geq -m$.
\end{lemma}
\begin{proof}
  The absolute case follows from~\cite[Proposition
  5.4]{Dugger_replacing} since $\cI$ has contractible classifying
  space. With Corollary~\ref{cor:fibrant-repl-of-descending-I-fib},
  the claim for $P$ descending $\cI$-fibrant follows from the absolute
  case.
\end{proof}

As another consequence of~\cite[Theorem 5.2]{Dugger_replacing}, we
note that the adjunction $\colim_{\cI} \colon \chIk \rightleftarrows
\chk \colon \mathrm{const}_{\cI}$ is a Quillen equivalence when
$\chIk$ is equipped with the absolute or descending $\cI$-model
structure. In particular, the composite of
\begin{equation}\label{eq:hocolim-of-const}
  (\mathrm{const}_{\cI}A)_{h\cI} \to \colim_{\cI}\mathrm{const}_{\cI}A
  \to A  
\end{equation}
is always a quasi-isomorphism, and each $P$ in $\chIk$ is related by a zig-zag
of $\cI$-equivalences
\begin{equation}\label{eq:const-repl}
\const_{\cI} \colim_{\cI} (P^{\mathrm{cof}}) \leftarrow
  P^{\mathrm{cof}} \to P
\end{equation}
to a constant $\cI$-diagram $\colim_{\cI} (P)^{\mathrm{cof}}$ where
$P^{\mathrm{cof}}  \to P$ is a cofibrant replacement. 

We record the following lemma for later use. 
\begin{lemma}\label{lem:hI-vs-products}
If $(P_{j})_{j \in J}$ is a family of $\cI$-chain complexes, then 
the canonical map 
\[ \left(\textstyle \prod_{j \in J} P_{j} \right)_{h\cI} \to
  \textstyle\prod_{j \in J} (P_{j})_{h\cI} \] 
is a quasi-isomorphism provided that all the $P_{j}$'s are descending
$\cI$-fibrant.  
\end{lemma}
\begin{proof}
  Arbitrary products of weak equivalences between fibrant objects in a
  model category are weak equivalences.  Therefore, using
  that~\eqref{eq:const-repl} is a zig-zag of $\cI$-equivalences
  between descending $\cI$-fibrant objects under our assumptions allows
  us to assume that each $P_{j}$ is of the form $\const_{\cI}
  A_{j}$. Forming the adjoint of the isomorphism $ \prod_{j \in
    J}\const_{\cI}\! A_{j}\!  \xrightarrow{\iso} \const_{\cI}\!
  \left(\textstyle \prod_{j \in J} \! A_{j} \right) $ under the
  Quillen equivalence $(\colim_{\cI},\mathrm{const}_{\cI})$ shows that
  the composite in

\[ 
\left(\textstyle \prod_{j \in J}\const_{\cI} A_{j}
\right)_{h\cI} \to \textstyle\prod_{j \in J} (\const_{\cI}
A_{j})_{h\cI} \xrightarrow{\sim} \textstyle \prod_{j \in J}
A_{j}
\]  
is a quasi-isomorphism.  Since the second map is a product of
quasi-isomorphisms, the claim follows by two-out-of-three.
\end{proof}

\subsection{Commutative \texorpdfstring{$\cI$}{I}-dgas}
Although essentially only our formulation of
Theorem~\ref{thm:AI-KI-adjunctions-intro} depends on the existence of
a lifted model structure on $\CchIk$, the following result is the main
motivation for working with commutative $\cI$-dgas.
\begin{theorem}\label{thm:chain-Quillen-equiv}
  The category $\CchIk$ admits a \emph{descending $\cI$-model structure}
  where a map is a weak equivalence (or fibration) if the underlying
  map in the descending $\cI$-model structure on $\chIk$ is. With this
  model structure, $\CchIk$ is Quillen equivalent to the category
  of $E_{\infty}$ dgas and to the category of commutative $Hk$-algebra
  spectra.
\end{theorem}
\begin{proof}
  The identification of $\cI$-diagrams with generalized symmetric
  spectra (see \cite[Proposition 9.1]{Richter-S_algebraic} or \cite[\S
  3.2]{Pavlov-S_sym-operads}) and \cite[Propositions 3.2.2 and
  3.3.1]{Pavlov-S_sym-operads} imply that $\chIk$ admits a
  \emph{positive $\cI$-model} structure with weak equivalences the
  $\cI$-equivalences.  This model structure has more cofibrations than
  the descending $\cI$-model structure and less cofibrations than the
  absolute $\cI$-model structure. The positive $\cI$-model structure
  lifts to $\chIk$ by~\cite[Theorem 4.1]{Pavlov-S_sym-operads}, and
  the resulting model category is Quillen equivalent to commutative
  $Hk$-algebra spectra and to $E_{\infty}$ dgas~\cite[Theorem~9.5, Corollary~8.3]{Richter-S_algebraic}. 
  
  The descending $\cI$-model structure on $\chIk$ is left proper,
  combinatorial, and has less cofibrations than the positive
  one. Hence the existence of the descending $\cI$-model structure on
  $\CchIk$ and the Quillen equivalence to the positive $\cI$-model
  structure immediately follow from~\cite[Theorem 11.3.2]{Hirschhorn_model}.
\end{proof}
The equivalence of homotopy categories resulting from this theorem is
actually induced by the homotopy colimit over $\cI$ with the $\cE$-action from
Theorem~\ref{thm:E-infinity-action-on-hocolim-I}:  
\begin{proposition}\label{prop:hI-equivalence-hty-categories}
The functor $(-)_{h\cI} \colon \CchIk \to \Echk$ induces an
equivalence of categories $\mathrm{Ho}(\CchIk) \to
\mathrm{Ho}(\Echk)$.  
\end{proposition}
\begin{proof}
An $\cI$-chain complex $X$ admits a \emph{bar resolution}
$\overline{X} \to X$ defined by $\overline{X}(\bld{n}) =
\hocolim_{\cI\downarrow\bld{n}}(X \circ \pi)$ where $\pi \colon
\cI\downarrow\bld{n} \to \cI$ is the canonical projection from the
overcategory forgetting the augmentation to $\bld{n}$. The inclusion
of the terminal object in $\cI\downarrow\bld{n}$ induces a map
$\overline{X} \to X$ which is a level equivalence by a homotopy
cofinality argument. The bar resolution has the property $\colim_{\cI}
\overline{X} \iso X_{h\cI}$. When $M$ is an $\cE$-algebra in $\chIk$,
then $\overline{M}$ inherits an $\cE$-algebra structure with diagonal
$\cE$-action (compare
Theorem~\ref{thm:E-infinity-action-on-hocolim-I} and the analogous
space level statement in~\cite[Lemma
6.7]{Schlichtkrull_Thom-symmetric}). When $M$ is a commutative
$\cI$-dga, then the $\cE$-algebra structure on $\colim_{\cI}(\overline{M})$
resulting from this observation and the strong monoidality of
$\colim_{\cI}$ coincides with the one on $M_{h\cI}$ provided by
Theorem~\ref{thm:E-infinity-action-on-hocolim-I}. We also note that if
$X$ is a cofibrant $\cI$-chain complex, then the map
$\colim_{\cI}\overline{X} \to \colim_{\cI}X$ is a
quasi-isomorphism. This can be checked directly on free $\cI$-chain
complexes, and the general case follows because both sides preserve
colimits and send generating cofibrations to levelwise injections.  

To prove the proposition, we note that the chain of Quillen
equivalences from Theorem~\ref{thm:chain-Quillen-equiv} sends a
commutative $\cI$-dga $M$ to $\colim_{\cI}M^{\mathrm{cof}}$, the
colimit over $\cI$ of a cofibrant replacement of $M$ in $\EchIk$. This
colimit is related to $M_{h\cI}$ by a natural zig-zag of $\cE$-algebra
maps   
\[ M_{h\cI} \leftarrow (M^{\mathrm{cof}})_{h\cI} \xrightarrow{\iso}
  \colim_{\cI}\overline{M^{\mathrm{cof}}} \to
  \colim_{\cI}{M^{\mathrm{cof}}}\] 
where the first map is a quasi-isomorphism since the cofibrant
replacement is an $\cI$-equivalence and the last map is a
quasi-isomorphism  by the above discussion since $M^{\mathrm{cof}}$ is
a cofibrant $\cI$-chain complex by~\cite[Theorem
4.4]{Pavlov-S_sym-operads}.  
\end{proof}

For later use we note that the commutative $\cI$-dga $\bC
F_{\bld{1}}^{\cI}(A)$ from~\eqref{eq:free-comm-I-dga-explicit} has the
following homotopical feature:  
\begin{lemma}\label{eq:free-comm-I-dga-on-cofibrant-acyclic}
  Let $A$ be a cofibrant acyclic chain complex. Then each $(\bC
  F_{\bld{1}}^{\cI}(A))(\bld{m})$ is cofibrant in $\chk$, and the unit
  $U^{\cI} \to \mathbb \bC 
  F_{\bld{1}}^{\cI}(A)$ is an absolute level equivalence.
\end{lemma}
\begin{proof}
This is an immediate consequence of the isomorphism \eqref{eq:F1IAm-identification}. 
\end{proof}

\section{Comparison of cochain functors}
We now define a simplicial $\cI$-chain complex $B^{\cI}_{\bullet}$ by
setting $B^{\cI}_p = A^{\cI}_p \boxtimes C^{\cI}_p$ in simplicial
level $p$ and using the $\boxtimes$-products of the simplicial
structure maps of $A^{\cI}$ and $C^{\cI}$ as simplicial structure maps
for $B^{\cI}_{\bullet}$. There is a natural
isomorphism  \begin{equation}\label{eq:BI-levelwise}B^{\cI}_p(\bld{m})
  = (A^{\cI}_{p}\boxtimes F^{\cI}_{\bld{0}}(C_p))(\bld{m}) \iso
  A^{\cI}_p (\bld{m}) \tensor C_p 
\end{equation}
that results from the definition of $\boxtimes$ as a left Kan extension.

The unit maps $U^{\cI} \to C^{\cI}$ and $U^{\cI} \to A^{\cI}$ induce a
chain
\begin{equation}\label{eq:comparison-simplicial-objects}
  A^{\cI}_{\bullet} \to B^{\cI}_{\bullet} \leftarrow C^{\cI}_{\bullet}
\end{equation}
of maps of simplicial objects in $\chIk$. By
Construction~\ref{constr:adj}, this chain gives rise to a chain of
natural transformations $A^{\cI} \to B^{\cI} \leftarrow C^{\cI}$ of
functors $(\sset)^{\op} \to \chIk$.
\begin{theorem}\label{thm:chain-complex-level-comparison}
  For every simplicial set $X$, the maps
  $A^{\cI}(X) \to B^{\cI}(X) \leftarrow C^{\cI}(X)$ are
  descending level equivalences between descending $\cI$-fibrant objects.
\end{theorem}
We prove the theorem at the end of the section. The definition of $B^{\cI}$ and our strategy of proof are motivated by the corresponding rational result in~\cite[\S 10]{FHT}.

\begin{corollary}\label{cor:hty-invariance-AI}
  If $X \to Y$ is a weak homotopy equivalence of simplicial sets, then
  $A^{\cI}(Y) \to A^{\cI}(X)$ is a descending level equivalence
  between descending $\cI$-fibrant objects.
\end{corollary}
\begin{proof}
  The map $C^{\cI}(Y) \to C^{\cI}(X)$ is an $\cI$-equivalence since
  $C(Y) \to C(X)$ is a quasi-isomorphism of chain complexes by the
  homotopy invariance of singular homology. By the theorem, the claim
  about  $A^{\cI}(Y) \to A^{\cI}(X)$ follows. 
\end{proof}

Combining Theorem~\ref{thm:chain-complex-level-comparison} with
Lemma~\ref{lem:hocolim-of-I-fibrant} does in particular imply that for
positive objects $\bld{m}$, the chain complex $A^{\cI}(X)(\bld{m})$
captures the cohomology groups of $X$ in degrees between $0$ and
$|\bld{m}|$. It should not be surprising that there is a functor from
spaces to chain complexes concentrated in degrees between $0$ and
$- |\bld{m}|$ which has this property: if one applies the smart
truncation $\tau_{\geq -m}$ degreewise to the simplicial object
$[p] \mapsto A_{\mathrm{PL},p}$ and applies
Construction~\ref{constr:hocolim_I} to the resulting simplicial object
$ \tau_{\geq -m}A_{\mathrm{PL},\bullet}$, one gets back
$\tau_{\geq -m}A_{\mathrm{PL}}$ since $\tau_{\geq -m}$ is right
adjoint. In view of this, the chain complexes  $A^{\cI}(X)(\bld{m})$
are analogous to truncations of $A_{\mathrm{PL}}(X)$. 

\begin{lemma}\label{lem:AI-BI-CI-level-equivalences}
  The maps in~\eqref{eq:comparison-simplicial-objects} are absolute
  level equivalences between absolute $\cI$-fibrant objects when
  evaluated in simplicial degree $p$.
\end{lemma}
\begin{proof}
  Let $\bld{m}$ be an object in $\cI$. By
  Lemma~\ref{eq:free-comm-I-dga-on-cofibrant-acyclic} the map $S^0 =
  U^{\cI}(\bld{m}) \to A^{\cI}_p(\bld{m})$ is a quasi-isomorphism
  between cofibrant and fibrant objects and thus even a chain homotopy
  equivalence. The map $S^0 \to C(\Delta^p) = C_{p}$ is a
  quasi-isomorphism by the known computation of
  $H^*(\Delta^p;k)$. Applying $F^{\cI}_{\bld{0}}$, it provides an
  absolute level equivalence $U^{\cI} \to
  C^{\cI}_p$. By~\eqref{eq:BI-levelwise}, we can decompose
  $U^{\cI}(\bld{m}) \to B^{\cI}_p(\bld{m})$ as \[S^0 \to
    C_p\xrightarrow{\iso} S^0 \tensor C_p \to A^{\cI}_p \tensor C_p .\]
  We already showed that the first map is a
  quasi-isomorphism. The last one is a quasi-isomorphism since
  $-\tensor C_p$ preserves chain homotopy equivalences. 
  The $\cI$-chain complexes $A^{\cI}_p$, $B^{\cI}_p$, and $C^{\cI}_p$
  are absolute $\cI$-fibrant for each $p\geq 0$ since they are absolute
  level equivalent to $U^{\cI}$ and $U^{\cI} = \const_{\cI}S^0$ is
  absolute $\cI$-fibrant.
\end{proof}

\begin{lemma}\label{lem:BI-contractible}
For all ${q \in \mathbb Z}$ and all positive objects $\bld{m}$ in $\cI$, the
simplicial $k$-module $B^{\cI}_{\bullet}(\bld{m})_q$ is
  contractible to $0$.
\end{lemma}
\begin{proof}
  From~\eqref{eq:BI-levelwise} we get an isomorphism
  $B^{\cI}_{\bullet}(\bld{m})_q \iso \bigoplus_{r+s=q} A^{\cI}_{\bullet}(\bld{m})_{r}\tensor C_{\bullet,s}$ of simplicial $k$-modules. Since
  $C_{\bullet, s}$ is contractible for every $s$ by
  Lemma~\ref{lem:properties-usual-cochains}(ii),
  so are the tensor
  products and thus also the sum.
\end{proof}

\begin{lemma}\label{lem:latching-map-pos-I-fibration}
  Let $D_{\bullet} \colon \Delta^{\op} \to \chIk$ be a simplicial
  object in $\cI$-chain complexes such that for all positive objects
  $\bld{m}$ in $\cI$ and all integers $q$ with
  $q > -|\bld{m}|$, the simplicial $k$-module
  $D_{\bullet}(\bld{m})_q$ is contractible to $0$. Then for all
  $p \geq 0$, the boundary inclusion $\partial \Delta^p \to \Delta^p$
  induces a descending level fibration
  $D( \Delta^p) \to D(\partial \Delta^p)$.
\end{lemma}
\begin{proof}
  A map in $\chIk$ is a descending level fibration if and only if it
  has the right lifting property against the maps
  $(U \to V) = F^{\cI}_{\bld{m}}(0 \to D^q)$ with $\bld{m}$ positive
  and $q > - |\bld{m}|$. By the adjunction~\eqref{eq:D-adjunction},
  the lifting property for $U\to V$ and
  $D( \Delta^p) \to D(\partial \Delta^p)$ is equivalent to the lifting
  property for $\partial \Delta^p \to \Delta^p$ and
  $K_D(V) \to K_D(U)$. Inspecting the definition of $K_D$, it follows
  that asking the latter lifting property for all $p\geq 0$ is
  equivalent to asking the map of simplicial sets
  $\chIk(V,D_{\bullet}) \to \chIk(U,D_{\bullet})$ to be an acyclic Kan
  fibration. Since $(F^{\cI}_{\bld{m}}, \mathrm{Ev}_{\bld{m}})$ is an
  adjunction and since morphisms in $\chk$ out of $D^q$ correspond to
  level $q$ elements, the assumption that $U \to V$ is
  $F^{\cI}_{\bld{m}}(0 \to D^q)$ implies that
  $\chIk(V,D_{\bullet}) \to \chIk(U,D_{\bullet})$ is isomorphic to
  $D_{\bullet}(\bld{m})_q \to 0$.  The source of this map is
  contractible by assumption and a Kan complex because it is the
  underlying simplicial set of a simplicial $k$-module. Hence
  $D_{\bullet}(\bld{m})_q \to 0$ is an acyclic Kan fibration.
\end{proof}
\begin{remark}
  When $D_{\bullet}(\bld{m})_q$ is not contractible,
  $D(\Delta^p)_q \to D(\partial \Delta^p)_q$ fails to be surjective.
  In view of Remark~\ref{rem:AI-not-contractible-in-I-degree-0}, this
  shows for example that
  $A^{\cI}(\Delta^p) \to A^{\cI}(\partial \Delta^p)$ is not an
  absolute level fibration.
\end{remark}

\begin{proposition}\label{prop:equivalence-on-extensions}
  Let $D_{\bullet} \!\to D'_{\bullet}$ be a natural transformation of
  functors ${\Delta^{\op} \!\to\! \chIk}$. Suppose that for all
  $p \geq 0$, the map $D_p \to D'_p$ is an absolute level equivalence
  between absolute $\cI$-fibrant objects and that for all positive
  objects $\bld{m}$ and all $q > -|\bld{m}|$, the simplicial
  $k$-modules $D_{\bullet}(\bld{m})_q$ and
  $D'_{\bullet}(\bld{m})_q$ are contractible. Then for every
  simplicial set~$X$, the map $D(X) \to D'(X)$ is a descending level
  equivalence between descending $\cI$-fibrant objects.
\end{proposition}
\begin{proof}
  As usual, this is proved by cell induction. Let us first assume that
  for all $p \geq 0$, the map $D(\partial \Delta^p) \to D'(\partial
  \Delta^p)$ is a descending level equivalence between descending
  $\cI$-fibrant objects. Any simplicial set $X$ can be written as a cell
  complex $X = \colim_{\lambda < \kappa}X_{\lambda}$ built from
  attaching cells of the form $\partial \Delta^p \to \Delta^p$. The
  functor $D$ takes the inclusion $\partial \Delta^p \to \Delta^p$ to
  a descending level fibration by
  Lemma~\ref{lem:latching-map-pos-I-fibration}. Since we assume that
  $D(\partial \Delta^p)$ and $D_p \iso D(\Delta^p)$ are descending
  $\cI$-fibrant, it follows from
  Corollary~\ref{cor:equivalence-fibrations-in-localization} that
  $D(\Delta^p) \to D(\partial \Delta^p)$ is a descending
  $\cI$-fibration. The same holds for $D'$. Since both $D$ and $D'$
  take colimits to limits by Lemma~\ref{lem:D-colimits-limits}, the
  coglueing lemma in the descending level model structure and the fact
  that base change preserves $\cI$-fibrations shows that $D(X) \to
  D'(X)$ arises as a limit of pointwise descending level equivalences
  between inverse systems of descending $\cI$-fibrations. Hence it is a
  descending level equivalence between descending $\cI$-fibrant objects.

  Since $\partial \Delta^p$ only has non-degenerate simplices in
  dimensions strictly less than $p$, an analogous induction over the
  dimension of $\partial \Delta^p$ shows the remaining claim about $D(\partial
  \Delta^p) \to D'(\partial \Delta^p)$.
\end{proof}

\begin{proof}[Proof of Theorem~\ref{thm:chain-complex-level-comparison}]
Combining Lemma~\ref{lem:AI-contractible},
Corollary~\ref{cor:properties-of-CI}(ii),  
Lemma~\ref{lem:BI-contractible}, and
Lemma~\ref{lem:AI-BI-CI-level-equivalences},  
the two maps $A^{\cI} \to B^{\cI}$ and $C^{\cI} \to B^{\cI}$ satisfy the
hypotheses
of Proposition~\ref{prop:equivalence-on-extensions}. 
\end{proof}

We can now also prove Theorem~\ref{thm:AI-KI-adjunctions-intro} from
the introduction:
\begin{proof}[Proof of Theorem~\ref{thm:AI-KI-adjunctions-intro}]
  The adjunction $(A^{\cI}, \langle -\rangle_\cI)$ arises from
  $A^{\cI}_{\bullet}$ by applying
  Construction~\ref{constr:adj}.
  Lemma~\ref{lem:AI-contractible} and
  Lemma~\ref{lem:latching-map-pos-I-fibration} show that $A^{\cI}$
  sends cofibrations to descending level fibrations and thus to descending
  $\cI$-fibrations. Corollary~\ref{cor:hty-invariance-AI} implies that
  $A^{\cI}$ sends weak homotopy equivalences to descending level
  equivalences and thus to $\cI$-equivalences.  The rest is an
  immediate consequence of the self-duality of model structures with
  respect to the passage to opposite categories and the adjunction
  isomorphisms~\eqref{eq:D-adjunction}.
\end{proof}

\subsection{The relation to polynomial forms}\label{subsec:A_PL} In
order to relate $A^{\cI}$ to the functor  $A_{\mathrm{PL}}$ used in
rational homotopy theory, we first identify the functor
$A_{\mathrm{PL},\bullet}$ described in~\eqref{eq:A_PL} as a two sided
bar construction. Arguing as in
Section~\ref{subsec:I-version-polynomial} and using the notation
$D^0_r$ introduced there, we get an isomorphism  
\[ B_{p}(S^0, \bC D^0, S^0) \iso \bC(D^0_{r_1(p)} \oplus \dots \oplus
  D^0_{r_p(p)}) \]
where $\bC$ denotes the free commutative dga on a chain complex. The assignments
\[ r_j(p) \mapsto \textstyle \sum_{0\leq l \leq j-1}
  t_l(p)\quad \text{ and }\quad t_j(p) \mapsto r_{j+1}(p) -r_j(p) \]
define inverse isomorphisms between $B_{p}(S^0,\bC D^0, S^0)$ and
$A_{\mathrm{PL},p}$, and these isomorphisms are compatible with the
structure maps described in Section~\ref{subsec:outline_constr}.

By adjunction, the canonical map
$D^0 \to (\const_{\cI} \bC D^0)(\bld{1})$ induces a map of commutative
$\cI$-dgas $\bC F^{\cI}_{\bld{1}}(D^0) \to \const_{\cI} \bC
D^0$. Using the above description of $A_{\mathrm{PL},\bullet}$ as a
two sided bar construction, this map in turn induces a  map
$A^{\cI}_{\bullet} \to \const_{\cI}A_{\mathrm{PL},\bullet}$ in
$\CchIk$ and thus a natural map
$A^{\cI}(X) \to \const_{\cI} A_{\mathrm{PL}}(X)$ on the extensions to
simplicial sets.
\begin{theorem} \label{thm:char0} Let $k$ be a field of characteristic
  $0$. Then $A^{\cI}(X) \to \const_{\cI} A_{\mathrm{PL}}(X)$ is a
  descending level equivalence. It induces a quasi-isomorphism
  $A^{\cI}(X)_{h\cI} \to A_{\mathrm{PL}}(X)$ that is an $\cE$-algebra
  map if we view the cdga $A_{\mathrm{PL}}(X)$ as an $\cE$-algebra
  by restricting along the canonical operad map from $\cE$ to the
  commutativity operad.
\end{theorem}
\begin{proof}
  In characteristic zero the homology groups of
  $(D^0)^{\otimes n}/\Sigma_n$ are isomorphic to the coinvariants
  $H_*(D^0)^{\otimes n}/\Sigma_n$ and the latter term is trivial for
  $n \geq 1$ because $D^0$ is acyclic. Therefore
  $\bC F^{\cI}_{\bld{1}}(D^0) \to \const_{\cI} \bC D^0 $ is an absolute
  level equivalence. The claim about general $X$ follows from
  Proposition~\ref{prop:equivalence-on-extensions} and the
  contractibility property of $A_{\mathrm{PL},\bullet}$ established
  in~\cite[Proposition 1.1]{Bousfield-G_rational}. Applying
  $(-)_{h\cI}$ to this descending level equivalence and composing with
  the natural quasi-isomorphism~\eqref{eq:hocolim-of-const} gives the
  quasi-isomorphism $A^{\cI}(X)_{h\cI} \to A_{\mathrm{PL}}(X)$. To see
  that it is an $\cE$-algebra map, we note that it follows from the
  definitions that~\eqref{eq:hocolim-of-const} is an $\cE$-algebra map
  when evaluated on a cdga.
\end{proof}

\section{Comparison of 
\texorpdfstring{$E_{\infty}$}{E-infinity}  structures}
Let $\cE$ be the Barratt--Eccles operad introduced in
Definition~\ref{def:Barratt-Eccles-operad}. We now define
$A\colon \sset^\op \to \Echk$ to be the composite
$A = (A^{\cI})_{h\cI}$ of the functor $A^{\cI}$ from the previous
section and the functor $(-)_{h\cI}\colon \CchIk \to \Echk$ resulting
from Theorem~\ref{thm:E-infinity-action-on-hocolim-I}. The following
proposition shows that $A$ is a \emph{cochain theory} in the sense
of~\cite{Mandell_cochain-multiplications}.

\begin{proposition} The functor $A\colon \sset^\op \to \Echk$ has the
  following properties.  
\begin{enumerate}[(i)]
\item It sends weak equivalence of simplicial sets to quasi-isomorphisms. 
\item For a sub-simplicial set $Y \subseteq X$, the induced map from
  $\mathrm{hofib}{(A(X/Y) \to A(*))}$ to $\mathrm{hofib}(A(X) \to A(Y))$
  is a quasi-isomorphism.
\item For a family $(X_{j})_{j \in J}$ of simplicial sets
  indexed by a set $J$, the canonical map  
$A(\coprod_{j \in J} X_{j}) \to \prod_{j \in J} A(X_{j})$ is a
quasi-isomorphism. 
\item It satisfies $H_0(A(*))\iso k$ and $H_n(A(*)) \iso 0$ if $n\neq 0$. 
\end{enumerate}
\end{proposition}
\begin{proof}
  Part (i) follows from Corollary~\ref{cor:hty-invariance-AI}, part
  (iv) is an immediate consequence of
  Theorem~\ref{thm:chain-complex-level-comparison}, and part (iii)
  follows from Lemma~\ref{lem:hI-vs-products} because $A^{\cI}$ takes
  coproducts in $\sset$ to products of fibrant objects in $\chIk$.

  For (ii), we view $X/Y$ as the pushout of $* \leftarrow Y \to
  X$. The functor $A^{\cI}$ sends this pushout to a pullback diagram
  displayed as the front face in the following cube:
\[\xymatrix@-1.5pc{
     &A^{\cI}(X/Y)'  \ar[rr] \ar@{->>}'[d][dd] && A^{\cI}(X)'\ar@{->>}[dd]\\
      A^{\cI}(X/Y) \ar[ur]^{\sim} \ar[rr] \ar@{->>}[dd] && A^{\cI}(X) \ar[ur]^{\sim}\ar@{->>}[dd]  \\
     &A^{\cI}(*)' \ar'[r][rr]  && A^{\cI}(Y)'\\
     A^{\cI}(*) \ar[ur]^{\sim} \ar[rr] && A^{\cI}(Y)  \ar[ur]^{\sim}
  }
\]
Here the vertical maps on the front are descending level fibrations
between descending $\cI$-fibrant objects by
Theorem~\ref{thm:AI-KI-adjunctions-intro}. The bottom face is obtained
by applying a fibrant replacement in the absolute $\cI$-model
structure to the map $A^{\cI}(*) \to A^{\cI}(Y)$, and the inwards
pointing arrows on the bottom are descending level equivalences by
Corollary~\ref{cor:fibrant-repl-of-descending-I-fib}.  The right hand
face is obtained by factoring $A^{\cI}(X) \to A^{\cI}(Y)'$ as an
acyclic cofibration $A^{\cI}(X) \to A^{\cI}(X)'$ followed by a
fibration $A^{\cI}(X)' \to A^{\cI}(Y)'$ in the absolute $\cI$-model
structure. Then $A^{\cI}(X) \to A^{\cI}(X)'$ is also a descending
level equivalence by
Corollary~\ref{cor:fibrant-repl-of-descending-I-fib}.  The last term
$A^{\cI}(X/Y)'$ is obtained by requiring the back face to be a
pullback. Right properness of the descending level model structure
implies that $A^{\cI}(X/Y) \to A^{\cI}(X/Y)'$ is a descending level
equivalence. It then follows from Lemma~\ref{lem:hocolim-of-I-fibrant}
that the square obtained by applying $(-)_{h\cI}$ to the front face is
quasi-isomorphic to the square obtained by evaluating the back face at
the object $\bld{0}$. The latter square is homotopy cartesian by
construction.
\end{proof}
Let $\cE^{\mathrm{cof}}$ be a cofibrant $E_{\infty}$ operad in the
sense of~\cite[Definition 4.2]{Mandell_cochain-multiplications}. Then
there exists an operad map $\cE^{\mathrm{cof}} \to \cE$ to the
Barratt--Eccles operad~\cite[Lemma
4.5]{Mandell_cochain-multiplications}, and by restricting along
$\cE^{\mathrm{cof}} \to \cE$ we may view $A$ as a functor to
$\cE^{\mathrm{cof}}$-algebras. On the other hand, the cosimplicial
normalization functor for the category $\Ecchk$ provided
by~\cite[Theorem 5.8]{Mandell_cochain-multiplications} allows one to
lift the ordinary cochain functor $C\colon \sset^\op \to \chk$ to a
functor with values in $\Ecchk$ (compare~\cite[\S
1]{Mandell_cochain-multiplications}). We are now in a situation
where~\cite[Main Theorem]{Mandell_cochain-multiplications} applies:
\begin{theorem}\label{thm:E-infty-comparison}
The functor $A \colon \sset^\op \to \Ecchk$ is naturally quasi-isomorphic
to the singular cochain functor $C \colon \sset^\op \to \Ecchk$. \qed
\end{theorem}

\begin{remark}
  It is well-known how to express the cup-$i$ products on the singular
  cohomology of spaces using the Barratt-Eccles operad, see for
  instance \cite[Theorem 2.1.1]{Berger-F_combinatorial}. This way the
  $\cE$-algebra structure on $A(X) = A^{\cI}(X)_{h\cI}$ gives rise to
  cup-$i$ products, and the previous theorem shows that they are
  equivalent to the usual cup-$i$ products on the cochain algebra.
\end{remark}

Theorem~\ref{thm:E-infty-comparison} also allows us to express Mandell's theorem~\cite{Mandell_cochains-homotopy-type} using $A^{\cI}$: 
\begin{proof}[Proof of Theorem~\ref{thm:mandell-intro}]
  Let $X$ and $Y$ be two finite type nilpotent spaces. By
  Proposition~\ref{prop:hI-equivalence-hty-categories}, the
  commutative $\cI$-dgas $A^{\cI}(X;\mathbb Z)$ and
  $A^{\cI}(Y;\mathbb Z)$ are $\cI$-equivalent in $\CchIZ$ if and only
  if $A^{\cI}(X;\mathbb Z)_{h\cI}$ and $A^{\cI}(Y;\mathbb Z)_{h\cI}$
  are quasi-isomorphic in $\EchZ$, which is in turn equivalent to
  being quasi-isomorphic in $\EcchZ$. By
  Theorem~\ref{thm:E-infty-comparison}, this holds if and only if
  $C^*(X;\mathbb Z)$ and $C^*(Y;\mathbb Z)$ are quasi-isomorphic in
  $\EcchZ$. By~\cite[Main Theorem]{Mandell_cochains-homotopy-type},
  this is the case if and only if $X$ and $Y$ are weakly equivalent.
\end{proof}



\end{document}